\documentclass{article}
\usepackage{graphicx}
\usepackage[left=3cm, top=2cm, right=3cm, bottom=3cm]{geometry}
\usepackage{listings}
\usepackage{float}
\usepackage{amssymb,amsmath,mathtools}
\usepackage{enumerate}
\usepackage[numbers,square,sort]{natbib}
\usepackage{wrapfig}
\usepackage{framed}
\usepackage{tabstackengine}
\usepackage{braket}
\usepackage[nottoc]{tocbibind}\settocbibname{References}
\usepackage{mathrsfs}

\usepackage{indentfirst}
\usepackage{setspace}
\usepackage{bm}

\usepackage[dvipsnames]{xcolor}
\usepackage{tikz}
\tikzset{every loop/.style={}}
\usetikzlibrary{arrows,shapes,decorations.markings,automata,backgrounds,petri,bending,calc}

\usepackage{fancyhdr}
\fancypagestyle{firststyle}
{
	
	\fancyfoot[l]{\footnotesize\textit{Date:} \today}
	\fancyfoot[c]{1}
	\fancyhead{}

}

\usepackage{tikz-cd} 
\tikzset{
    labl/.style={anchor=south, rotate=90, inner sep=.5mm}
}

\usepackage{enumitem}
\setlist[enumerate,1]{label=(\arabic*)}
\newlist{steplist}{enumerate}{1}
\setlist[steplist]{label={Step \arabic*:}, ref={Step \arabic*}}

\usepackage[utf8]{inputenc}
\usepackage[english]{babel} 

\usepackage{amsthm}
\newtheorem{thm}{Theorem}[section]
\newtheorem{lem}[thm]{Lemma}
\newtheorem{prop}[thm]{Proposition}
\newtheorem{cor}[thm]{Corollary}
\numberwithin{equation}{section}

\theoremstyle{definition}
\newtheorem{defn}[thm]{Definition} 
\newtheorem{remk}[thm]{Remark}

\newcommand{\cC}{\mathcal{C}}

\newcommand{\cH}{\mathcal{H}}

\newcommand{\cM}{\mathcal{M}}
\newcommand{\cN}{\mathcal{N}}

\newcommand{\cP}{\mathcal{P}}

\newcommand{\fa}{\mathfrak{a}}
\newcommand{\fb}{\mathfrak{b}}
\newcommand{\fc}{\mathfrak{c}}
\newcommand{\h}{\mathfrak{h}}

\newcommand{\hh}{\hat{\mathfrak{h}}}
\newcommand{\hH}{\hat{\mathcal{H}}}

\newcommand{\diam}{\operatorname{diam}}
\newcommand{\Hull}{\operatorname{Hull}}

\makeatletter
\newsavebox{\@brx}
\newcommand{\llangle}[1][]{\savebox{\@brx}{\(\m@th{#1\langle}\)}%
	\mathopen{\copy\@brx\kern-0.5\wd\@brx\usebox{\@brx}}}
\newcommand{\rrangle}[1][]{\savebox{\@brx}{\(\m@th{#1\rangle}\)}%
	\mathclose{\copy\@brx\kern-0.5\wd\@brx\usebox{\@brx}}}
\makeatother

\begin{document}
\begin{center}
{\LARGE\bf
Semistability of cubulated groups}\\
\bigskip
\bigskip
{\large Sam Shepherd}
\end{center}
\bigskip

\begin{abstract}
	We prove that all cubulated groups are semistable at infinity.
	In doing so we prove two further results about cubulations of groups.
	The first of these states that any one-ended cubulated group has a cubulation for which all halfspaces are one-ended.
	The second states that any cubulated group has a cubulation for which all quarterspaces are deep -- analogous to the fact that passing to the essential core of a given cubulation ensures that all halfspaces are deep.
\end{abstract}
\bigskip
\tableofcontents

\bigskip
\section{Introduction}

A connected, locally finite CW complex $X$ is \emph{semistable at infinity} if any two proper rays $r,s:[0,\infty)\to X$ converging to the same end are properly homotopic.
The terminology comes from the following connection with inverse systems of groups. An inverse systems of groups $\{H_n\}$ is \emph{semistable} if, for each $n$, the images of the bonding homomorphisms $H_m\to H_n$ are the same for all but finitely many $m>n$.
Given a connected, locally finite CW complex $X$, we consider the inverse system of groups $\{\pi_1(X-C_n,r)\}$, where $\{C_n\}$ is an exhausting sequence of compact subsets of $X$, $r$ is a proper ray in $X$, and the bonding maps are induced by inclusions of subsets.
If $X$ is one-ended and simply connected then $\{\pi_1(X-C_n,r)\}$ is semistable if and only if $X$ is semistable at infinity \cite{Mihalik83}.
In this case the inverse limit of $\{\pi_1(X-C_n,r)\}$ provides a well defined notion of \emph{fundamental group at infinity} for $X$. Both semistability at infinity and the fundamental group at infinity are quasi-isometry invariants for simply connected, locally finite CW complexes \cite{Brick93,Geoghegan08}.

If a finitely presented group $G$ acts properly and cocompactly on a simply connected CW complex $X$, then we say that $G$ is \emph{semistable at infinity} if $X$ is semistable at infinity.
The quasi-isometry invariance from the preceding paragraph implies that semistability of $G$ is independent of the choice of complex $X$, and is a quasi-isometry invariant for groups.
Various classes of groups are known to be semistable at infinity, including hyperbolic groups \cite{Krasinkiewicz77,GeogheganKrasinkiewicz91,GeogheganSwenson19,Bowditch99,Levitt98,Swarup96,BestvinaMess91}, Artin and Coxeter groups \cite{Mihalik96}, one-relator groups \cite{MihalikTschanz92a}, and certain graphs of groups \cite{MihalikTschantz92b}.
It is unknown if all CAT(0) groups are semistable at infinity -- and this is one of the more heavily studied problems in the field (see \cite{GeogheganSwenson19}). In fact it is even unknown if all finitely presented groups are semistable at infinity.
In this paper we prove the following.

\begin{thm}\label{thm:semistable}
	Cubulated groups are semistable at infinity.
\end{thm}

There is also a connection between semistability and group cohomology: if $G$ is semistable at infinity then $H^2(G,\mathbb{Z}G)$ is free abelian \cite{Houghton77,GeogheganMihalik85}, so we obtain the following corollary.
It is an open question of Hopf whether $H^2(G,\mathbb{Z}G)$ is free abelian for all finitely presented groups.

\begin{cor}
	If $G$ is a cubulated group then $H^2(G,\mathbb{Z}G)$ is free abelian.
\end{cor}

We say that a group $G$ is \emph{cubulated} if it acts properly and cocompactly on a CAT(0) cube complex $X$ - and we refer to such an action as a \emph{cubulation} of $G$.
Examples of cubulated groups include small cancellation groups, finite volume hyperbolic 3-manifold groups and many Coxeter groups -- see \cite{WiseRiches} for a more extensive list.
The geometry and combinatorics of CAT(0) cube complexes is a rich and dynamic theory, and it grants cubulated groups with many properties stronger than those of CAT(0) groups -- such as bi-automaticity \cite{Swiatkowski06} and the Tits Alternative \cite{SageevWise05}.

We remark that semistability at infinity for many cubulated groups, including virtually special groups (see \cite{HaglundWise08}), can be deduced fairly directly from existing literature. Indeed, if $X$ is a finite non-positively curved cube complex whose hyperplanes are two-sided and do not self-intersect, then successively cutting $X$ along hyperplanes corresponds to successively splitting the fundamental group of $X$, terminating in trivial groups. We deduce from \cite{MihalikTschantz92b} that the fundamental group of $X$ is semistable at infinity, and it follows that all virtually special groups are semistable at infinity because semistability is a quasi-isometry invariant.
However, for general cubulated groups there is no (virtual) hierarchy that we can use, and indeed our proof of Theorem \ref{thm:semistable} uses a different argument. 

The idea for our proof is as follows.
Using results from the literature we can easily reduce to the case of a one-ended group; and if the cubulation is given by a CAT(0) cube complex $X$ then we can reduce to showing that, for any compact $C\subset X$, any loop sufficiently far from $C$ (and based on a proper base ray) can be pushed arbitrarily far from $C$ (relative to the base ray) by a homotopy that avoids $C$.
The key step is to achieve this ``pushing out'' using two geometric (and cubical) properties of $X$. However, these properties do not hold for arbitrary cubulations, so we must first modify the cubulation (Theorem \ref{thm:halfquarter}).
In fact most of the work in this paper goes into proving Theorem \ref{thm:halfquarter}.
The first geometric property we need is one-ended halfspaces, which is obtained with the following theorem.

\begin{thm}\label{thm:halfspaces}
	Let $G$ be a group acting cocompactly on a one-ended locally finite CAT(0) cube complex $X$. Suppose there exists a subgroup $\Gamma<G$ whose induced action on $X$ is proper and cocompact. Then there is a locally finite CAT(0) cube complex $Y$ with the following properties:
	\begin{enumerate}
		\item All halfspaces in $Y$ are one-ended.
		\item\label{item:Gcocompact} $G$ acts cocompactly on $Y$.
		\item There exists a $G$-equivariant quasi-isometry $\theta:X\to Y$.
		\item\label{item:halfspacestabs} The $G$-stabilizers of hyperplanes in $Y$ are subgroups of the $G$-stabilizers of hyperplanes in $X$.
	\end{enumerate}	
\end{thm}

The second geometric property we need is about quarterspaces.
A \emph{quarterspace} in a CAT(0) cube complex is an intersection $\h_1\cap\h_2$ of transverse halfspaces (equivalently, halfspaces for which the corresponding hyperplanes intersect). The quarterspace $\h_1\cap\h_2$ is \emph{shallow} if it is contained in a bounded neighborhood of the opposite quarterspace $\h_1^*\cap\h_2^*$, otherwise $\h_1\cap\h_2$ is \emph{deep}.
These notions are analogous to the notions of halfspaces being shallow or deep (see Section \ref{subsec:CCC}). 
One can remove shallow halfspaces using the essential core of Caprace-Sageev (Proposition \ref{prop:essential}), and similarly the following theorem provides a way to remove shallow quarterspaces.

\begin{thm}\label{thm:quarterspaces}
	Let $G$ be a group acting cocompactly on a CAT(0) cube complex $X$. Then there is a CAT(0) cube complex $Y$ with the following properties:
	\begin{enumerate}
		\item All quarterspaces in $Y$ are deep.
		\item $G$ acts cocompactly on $Y$.
		\item There exists a $G$-equivariant quasi-isometry $\phi:Y\to X$.
		\item\label{item:hspace} $\phi$ maps each halfspace in $Y$ to within finite Hausdorff distance of a halfspace in $X$.
		\item $Y$ is locally finite if $X$ is locally finite.
	\end{enumerate}	
\end{thm}

Applying Theorem \ref{thm:halfspaces} and then Theorem \ref{thm:quarterspaces} to the case where $G$ acts properly as well as cocompactly yields the following theorem, which we use in the proof of Theorem \ref{thm:semistable} as outlined above.
Note that the halfspaces remain one-ended when applying Theorem \ref{thm:quarterspaces} because of property \ref{item:hspace}.

\begin{thm}\label{thm:halfquarter}
	Every one-ended cubulated group admits a cubulation in which all halfspaces are one-ended and all quarterspaces are deep.
\end{thm}

\begin{remk}\label{remk:deeper}
	One can alternatively use panel collapse \cite{HagenTouikan19} to show that any cocompact action on a CAT(0) cube complex can be modified to make quarterspaces deep. Indeed, panel collapse yields a hyperplane-essential action, and one can easily deduce that the cube complex has no shallow quarterspaces in this case.
	However, there is no analogue to property \ref{item:hspace} from Theorem \ref{thm:quarterspaces} when performing panel collapse, so this does not give an alternative proof of Theorem \ref{thm:halfquarter}.
	
	In fact, a cocompact, essential and hyperplane-essential action on a locally finite CAT(0) cube complex satisfies a stronger version of having deep quarterspaces in the sense that no quarterspace $\h_1\cap\h_2$ is contained in a bounded neighborhood of its complement $(\h_1\cap\h_2)^*=\h_1^*\cup\h_2^*$. Indeed, suppose the cube complex is a product $X=X_1\times\cdots X_n$ of irreducible cube complexes; if the halfspaces $\h_1,\h_2$ come from different factors $X_i$ then we can use the fact that $\h_1,\h_2$ are deep in their respective factors to deduce that $\h_1\cap\h_2$ is not contained in a bounded neighborhood of its complement, and if $\h_1,\h_2$ come from the same factor $X_i$ then we can use \cite[Proposition 1]{Hagen22}.
\end{remk}

The idea for the proof of Theorem \ref{thm:halfspaces} is to take a halfspace $\h_0$ with more than one end and chop it up using a $G_{\h_0}$-orbit of finite subcomplexes. We then obtain a new CAT(0) cube complex $Y$ by replacing the halfspace $\h_0$ with new halfspaces that correspond to the pieces leftover from the chopping, and we do the same for every $G$-translate of $\h_0$. Formally speaking, $Y$ is constructed from the cubing of a certain pocset (see Section \ref{subsec:pocsets} for background on pocsets). The new cube complex $Y$ might still have halfspaces with more than one end, but if we iterate the construction then we argue that it must terminate after a finite number of steps by considering the accessibility of the $\Gamma$-stabilizers of hyperplanes. The halfspaces in the terminal cube complex have at most one end, and we can easily remove bounded halfspaces using the essential cores of Caprace-Sageev.
The assumption in Theorem \ref{thm:halfspaces} about the existence of the subgroup $\Gamma<G$ is needed in the proof when we consider the accessibility of the $\Gamma$-stabilizers of hyperplanes, but we conjecture that Theorem \ref{thm:halfspaces} holds without this assumption.

The proof of Theorem \ref{thm:quarterspaces} is similar to that of Theorem \ref{thm:halfspaces} in that it involves cubing a pocset to obtain a new cube complex with lower complexity, and then iterating until the desired cube complex $Y$ is obtained. In this case the idea behind the pocset is to take certain shallow quarterspaces $\h_1\cap\h_2$ and pull apart the halfspaces $\h_1$ and $\h_2$. And this time the measure of complexity is just the number of $G$-orbits of vertices rather than anything to do with accessibility.

The structure of the paper is as follows.
In Section \ref{sec:prelim} we provide some background on CAT(0) cube complexes, pocsets, cubings, group splittings and accessibility of groups.
We prove Theorems \ref{thm:halfspaces}, \ref{thm:quarterspaces} and \ref{thm:semistable} in Sections \ref{sec:oneend}, \ref{sec:quarterspaces} and \ref{sec:semistability} respectively.
Finally, in Section \ref{sec:example} we give an example of a one-ended group with a cubulation given by a CAT(0) cube complex that is essential and contains an infinite-ended halfspace.
This demonstrates that Theorem \ref{thm:halfspaces} is not vacuous, and that it requires more than simply passing to the essential core of a CAT(0) cube complex.

\textbf{Acknowledgments:}\,
I thank Michael Mihalik for suggesting the problem of semistability of cubulated groups to me in the first place, and for his helpful comments on the paper.
I also thank Elia Fioravanti, Nicholas Touikan and Michah Saageev for their comments -- in particular Elia's comments lead to Remark \ref{remk:deeper}.
And I thank the referee for their helpful comments and corrections.

\bigskip
\section{Preliminaries}\label{sec:prelim}

\subsection{CAT(0) cube complexes}\label{subsec:CCC}

In this section we recall some basic concepts and facts regarding CAT(0) cube complexes. See \cite{WiseRiches,Manning20} for further background and proofs, including the definitions of CAT(0) cube complex and halfspace.

Let $X$ be a CAT(0) cube complex. We will write $\cH=\cH(X)$ for the collection of halfspaces. In this paper halfspaces will be combinatorial, so we will consider $\h\in\cH$ as a collection of vertices in $X$ rather than a convex subspace of $X$. We will write $\h^*$ for the complementary halfspace, so $\h\sqcup\h^*$ is a partition of the vertex set $X^0$.
We will write $\hat{\h}$ for the hyperplane corresponding to $\h$ and $\hH(X)$ for the collection of hyperplanes in $X$. 

We will mostly work with the combinatorial metric on $X$, i.e. the metric induced by the 1-skeleton $X^1$. We will denote this metric by $d$. Occasionally we will refer to the CAT(0) metric on $X$ (i.e. the metric induced by making every cube unit Euclidean), but we will always make this explicit.
(Note that properties related to these two metrics often coincide, for instance a subcomplex $Y\subseteq X$ is convex in the CAT(0) metric if and only if $Y^1$ is convex in $X^1$. These two metrics are also bi-Lipschitz equivalent if $X$ is finite dimensional.)
When describing properties of sets of vertices we will tacitly be referring to their induced subgraphs in $X^1$; for example we will say that $\fa\subseteq X^0$ is \emph{convex} if its induced subgraph is convex in $X^1$, or $\fb\subseteq\fa$ \emph{separates} $\fa$ if the induced subgraph of $\fb$ separates the induced subgraph of $\fa$.
The \emph{convex hull} $\Hull(\fa)\subseteq X^0$ of $\fa\subseteq X^0$ is the smallest convex set containing $\fa$ -- equivalently $\Hull(\fa)$ is the intersection of all halfspaces containing $\fa$ (with the empty intersection being $X^0$ by convention).
We will also use the following definition.

\begin{defn}\label{defn:Rthickening}
	Following \cite[\S4a]{HaglundWise12}, for $S\subseteq X$ define the \emph{cubical neighborhood} $N(S)\subseteq X^0$ to be the collection of vertices of cubes that intersect $S$. For an integer $R\geq0$, define the \emph{cubical $R$-thickening} $S^{+R}$ inductively by setting $S^{+0}$ to be the 0-skeleton of the smallest subcomplex of $X$ that contains $S$, and $S^{+(R+1)}:=N(S^{+R})$.
	We will often be interested in cubical neighborhoods and thickenings of hyperplanes $\hh$; note that $N(\hh)=\hh^{+0}$ is convex and consists of the endpoints of edges that join $\h$ to $\h^*$.
\end{defn}

\begin{remk}\label{remk:thickconvex}
	If $\fa\subseteq X^0$ is convex then $\fa^{+R}$ is convex for all $R\geq0$ -- in particular $R$-thickenings of hyperplanes are convex. 
\end{remk}

\begin{remk}\label{remk:thickneigh}
	Cubical thickenings are not the same as metric neighborhoods, but they are related because 
	$$\cN_R(\fa)\subseteq\fa^{+R}\subseteq\cN_{R\dim X}(\fa)$$
	for all $\fa\subseteq X^0$, where $\cN_r(\fa)\subset X^0$ denotes the $r$-neighborhood of $\fa$.
\end{remk}

\begin{defn}\label{defn:essential}
	A halfspace $\h$ is \emph{shallow} if it is contained in a bounded neighborhood of $\h^*$, otherwise $\h$ is \emph{deep}. A CAT(0) cube complex is \emph{essential} if all of its halfspaces are deep.
\end{defn}

The following proposition is due to Caprace--Sageev, it is a special case of \cite[Proposition 3.5]{CapraceSageev11}.

\begin{prop}\label{prop:essential}
	If $G$ is a group that acts cocompactly on an unbounded CAT(0) cube complex $X$, then there is a $G$-invariant convex subspace $Y\subseteq X$, which is either a subcomplex or a finite intersection of hyperplanes, and $Y$ is essential with respect to its induced cube complex structure. We call $Y$ the essential core of $X$.
\end{prop}

We now recall the notion of quarterspace, and the concept of a quarterspace being shallow or deep. These notions are new to this paper, but are analogous to Definition \ref{defn:essential}.

\begin{defn}	
	Halfspaces $\h_1,\h_2$ are \emph{transverse} if $\h_1\cap\h_2,\h_1\cap\h^*_2,\h^*_1\cap\h_2,\h^*_1\cap\h^*_2$ are all non-empty (equivalently if we get a non-empty intersection of the bounding hyperplanes $\hh_1\cap\hh_2\neq\emptyset$).
	In this case $\h_1\cap\h_2,\h_1\cap\h^*_2,\h^*_1\cap\h_2,\h^*_1\cap\h^*_2$ are referred to as \emph{quarterspaces}.
	
	A quarterspace $\h_1\cap\h_2$ is \emph{shallow} if it is contained in a bounded neighborhood of $\h^*_1\cap\h^*_2$, otherwise $\h_1\cap\h_2$ is \emph{deep}. The \emph{depth} of a quarterspace $\h_1\cap\h_2$ is the least integer $r$ such that $\h_1\cap\h_2$ is contained in the $(r+2)$-neighborhood of $\h^*_1\cap\h^*_2$, with $r=\infty$ if $\h_1\cap\h_2$ is deep. Note that $r\geq0$ because any path from $\h_1\cap\h_2$ to $\h^*_1\cap\h^*_2$ has length at least 2.
\end{defn}

\begin{remk}
	By considering projection maps, one can show that a quarterspace $\h_1\cap\h_2$ is shallow if and only if it is contained in a bounded neighborhood of $\hat{\h}_1\cap\hat{\h}_2$. Furthermore, the depth of $\h_1\cap\h_2$ is the least integer $r$ such that $\h_1\cap\h_2$ is contained in the $r$-neighborhood of $N(\hh_1)\cap N(\hh_2)$.
\end{remk}

The following three lemmas are well known, and will be used throughout the paper.

\begin{lem}\label{lem:d}
	If $\fa_1,\fa_2\subseteq X^0$ are convex then $d(\fa_1,\fa_2)$ is equal to the number of halfspaces $\h\in\cH(X)$ with $\fa_1\subseteq\h$ and $\fa_2\subseteq\h^*$.
	By considering $\fa_1=\{x_1\}$ and $\fa_2=\{x_2\}$ for $x_1,x_2\in X^0$, we deduce that any geodesic in $X^1$ has edges dual to distinct hyperplanes.
\end{lem}

\begin{lem}\label{lem:intersect}
	Any finite collection of pairwise intersecting convex sets in $X^0$ has non-empty intersection.
\end{lem}

\begin{lem}\label{lem:finhull}
	If $\fa\subseteq X^0$ is finite then $\Hull(\fa)$ is finite.
\end{lem}

\subsection{Pocsets and cubings}\label{subsec:pocsets}

We now recall the construction of a CAT(0) cube complex from a pocset or wallspace. This construction was originally due to Sageev \cite{Sageev95}, although our formulation will be closer to the version in \cite{Manning20} -- see also \cite{Nica04,ChatterjiNiblo05,Roller16,WiseRiches}.

\begin{defn}
	A \emph{pocset} is a poset $(\cP,\leq)$ together with an involution $A\mapsto A^*$ for all $A\in\cP$ satisfying:
	\begin{enumerate}
		\item $A$ and $A^*$ are incomparable.
		\item $A\leq B \Rightarrow B^*\leq A^*$.
	\end{enumerate}
We define $\hat{\cP}$ to be the set of pairs $\{A,A^*\}$.
Elements $A,B\in\cP$ are \emph{transverse} if neither $A$ nor $A^*$ is comparable with $B$. The \emph{width} of $\cP$ is the maximum number of pairwise transverse elements, if such a maximum exists, otherwise we say the width is $\infty$.
\end{defn}

\begin{defn}
An \emph{ultrafilter} on a pocset $\cP$ is a subset $\omega\subseteq \cP$ satisfying:
\begin{enumerate}
	\item (Completeness) For every $A\in \cP$, exactly one of $\{A,A^*\}$ is in $\omega$.
	\item (Consistency) If $A\in\omega$ and $A\leq B$, then $B\in\omega$.
\end{enumerate}
An ultrafilter $\omega$ is \emph{DCC} (descending chain condition) if it contains no strictly descending infinite chain $A_1>A_2>A_3>\cdots$.
\end{defn}

\begin{prop}\label{prop:cubing}
	Let $\cP$ be a pocset of finite width that admits at least one DCC ultrafilter. Then there is a CAT(0) cube complex $C=C(\cP)$, called the \textbf{cubing of $\cP$}, such that:
	\begin{enumerate}
		\item The vertices of $C$ are the DCC ultrafilters on $\cP$.
		\item Two vertices $\omega_1,\omega_2$ in $C$ are joined by an edge if and only if $\omega_1\triangle\omega_2=\{A,A^*\}$ for some $A\in \cP$.
		\item\label{item:halfspace} The halfspaces of $C$ take the form $\{\omega\mid A\in\omega\}$ for $A\in\cP$, so we have a pocset isomorphism $(\cP,\leq)\cong(\cH(C),\subseteq)$, which also induces an identification $\hat{\cP}\cong\hH(C)$.
		\item The dimension of $C$ is equal to the width of $\cP$.
	\end{enumerate} 
\end{prop}

\begin{defn}\label{defn:wallspace}
	A \emph{wallspace} $(X,\cP)$ is a set $X$ together with a family $\cP$ of non-empty subsets that is closed under complementation, such that for any $x,y\in X$ the set $\{A\in \cP\mid x\in A,\,y\notin A\}$ is finite. $\cP$ forms a pocset under inclusion, with the involution $A\mapsto A^*$ given by complementation. Moreover, for any $x\in X$ the set
	$$\omega_x:=\{A\in\cP\mid x\in A\}$$
	is a DCC ultrafilter. Therefore, if $\cP$ has finite width, we can form the cubing $C=C(\cP)$, and we have a map $X\to C^0$.
\end{defn}

Any CAT(0) cube complex $X$ forms a wallspace $(X^0,\cH(X))$. We then get the following duality theorem between finite dimensional CAT(0) cube complexes and finite width pocsets.

\begin{thm}\label{thm:duality}
	If $X$ is a finite dimensional CAT(0) cube complex then $\cH(X)$ has finite width, and the map $X^0\to C(\cH(X))$ extends to an isomorphism of cube complexes $X\cong C(\cH(X))$.
\end{thm}

Many geometric features of cubings can be interpreted using the pocset, for example distances, adjacent vertices and cocompactness.

\begin{lem}\label{lem:domega}
	$d(\omega_1,\omega_2)=\tfrac{1}{2}|\omega_1\triangle\omega_2|=|\omega_1-\omega_2|$ for $\omega_1,\omega_2\in C(\cP)^0$.
\end{lem}

\begin{lem}\label{lem:edges}
	For $\omega\in C(\cP)^0$, the vertices adjacent to $\omega$ are precisely the ultrafilters of the form $(\omega\cup\{A^*\})-\{A\}$ for $A\in\omega$ that is $\leq$-minimal in $\omega$.
\end{lem}

\begin{lem}\label{lem:cocompact}
	If a group $G$ acts on $\cP$, then the induced action on $C(\cP)$ is cocompact if and only if there are finitely many $G$-orbits of collections of pairwise transverse elements of $\cP$.
\end{lem}
 
We will also make use of the following (nonstandard) definition and lemmas.

\begin{defn}
	A \emph{partial ultrafilter} on a pocset $\cP$ is a subset $\omega\subseteq \cP$ (possibly empty) such that if $A\in\omega$ and $A\leq B$ then $B^*\notin\omega$. (Note that $\omega$ contains at most one element from each pair $\{A,A^*\}$.)
	We will sometimes refer to ultrafilters as \emph{complete ultrafilters} to stress that they satisfy the completeness property, which is what distinguishes ultrafilters from partial ultrafilters.
A partial ultrafilter $\omega$ is \emph{DCC} (descending chain condition) if it contains no strictly descending infinite chain $A_1>A_2>A_3>\cdots$. A partial ultrafilter $\omega$ is \emph{cofinite} if $\omega\cap\{A,A^*\}$ is empty for only finitely many $A\in\cP$.
\end{defn}

\begin{lem}\label{lem:extend}
	Any partial ultrafilter $\omega$ can be extended to a complete ultrafilter $\bar{\omega}$. Moreover, if $\omega$ is DCC and cofinite, then $\bar{\omega}$ is DCC.
\end{lem}
\begin{proof}
Take $A\in\cP$ with $\omega\cap\{A,A^*\}=\emptyset$. At least one of $\omega\cup\{A\}$ or $\omega\cup\{A^*\}$ must be a partial ultrafilter: indeed otherwise there exist $A\leq B_1$ and $A^*\leq B_2$ with $B_1^*,B_2^*\in\omega$, and this implies $B_1^*\leq B_2$, contradicting the fact that $\omega$ is a partial ultrafilter. The union of a chain of partial ultrafilters is clearly a partial ultrafilter, so it follows from Zorn's lemma that $\omega$ can be extended to an ultrafilter $\bar{\omega}$. If $\omega$ is cofinite, then any strictly descending infinite chain in $\bar{\omega}$ contains an infinite subchain in $\omega$; hence $\bar{\omega}$ is DCC if $\omega$ is DCC and cofinite.
\end{proof}

\begin{lem}\label{lem:partial=subcomplex}
	Let $\cP$ be a pocset of finite width that admits at least one DCC ultrafilter.
	Let $\omega$ be a DCC partial ultrafilter on $\cP$ such that if $A\in\omega$ and $A\leq B$, then $B\in\omega$.
	Then for each $\omega_0\in C(\cP)^0$, $\omega$ can be extended to a DCC complete ultrafilter given by
	\begin{equation}\label{baromega}
\bar{\omega}:=\omega\sqcup\{A\in\omega_0\mid \omega\cap\{A,A^*\}=\emptyset\}.
	\end{equation}
	Moreover, the set of all possible DCC complete extensions $\bar{\omega}$ of $\omega$ is equal to the intersection of the halfspaces $\{\bar{\omega}\in C(\cP)^0\mid A\in\bar{\omega}\}$ for $A\in\omega$, so it forms a convex subcomplex of $C(\cP)$. (If $\omega=\emptyset$ then this intersection is $C(\cP)^0$ by convention.)
\end{lem}
\begin{proof}
	The completeness axiom is clearly satisfied by $\bar{\omega}$. To check consistency, suppose $A\in\bar{\omega}$ and $A\leq B$; we wish to show that $B\in\bar{\omega}$. We have two cases. In the first case $A\in\omega$, so $B\in\omega\subseteq\bar{\omega}$ by hypothesis of $\omega$. In the second case $A\in\omega_0$ and $\omega\cap\{A,A^*\}=\emptyset$; observe that $B^*\notin\omega$, otherwise we would have $B^*\leq A^*$ and $A^*\in\omega$; so either $B\in\omega\subseteq\bar{\omega}$, or $\omega\cap\{B,B^*\}=\emptyset$ and $B\in\omega_0$ by consistency of $\omega_0$, and we again have $B\in\bar{\omega}$. So $\bar{\omega}$ is a complete ultrafilter.
	Furthermore, we deduce that $\bar{\omega}$ is DCC because any strictly descending infinite chain in $\bar{\omega}$ contains an infinite subchain in either $\omega$ or $\omega_0$.
\end{proof}

\subsection{Group splittings and accessibility}

By a \emph{splitting} of a group $G$ we mean an action on a tree $T$ without edge inversions. The splitting is \emph{finite} if the action is cocompact and \emph{non-trivial} if there is no fixed point. The splitting is \emph{over finite subgroups} if the edge groups are finite.
A finitely generated group is \emph{accessible} if it admits a splitting over finite subgroups in which each vertex group is either finite or one-ended. For such a splitting the vertex groups do not themselves admit non-trivial splittings over finite subgroups \cite{Stallings71}.
(One can also assume that the splitting is finite by passing to a minimal invariant subtree.)
Dunwoody proved the following result.

\begin{thm}\cite{Dunwoody85}\label{thm:accessibility}
	Every finitely presented group is accessible.
\end{thm}

We can then deduce the following theorem using \cite[Theorem 5.12]{Dicks80}.

\begin{thm}\label{thm:terminate}
	Let $G$ be a finitely presented group and let $(G_i)$ be a sequence of groups such that $G_0=G$, and $G_{i+1}$ is a vertex group in some non-trivial finite splitting of $G_i$ over finite subgroups. Then the sequence $(G_i)$ terminates.
\end{thm}

\bigskip
\section{Reducing to one-ended halfspaces}\label{sec:oneend}

In this section we prove the following theorem.

\theoremstyle{plain}
\newtheorem*{thm:halfspaces}{Theorem \ref{thm:halfspaces}}
\begin{thm:halfspaces}
	Let $G$ be a group acting cocompactly on a one-ended locally finite CAT(0) cube complex $X$. Suppose there is a subgroup $\Gamma<G$ whose induced action on $X$ is proper and cocompact. Then there is a locally finite CAT(0) cube complex $Y$ with the following properties:
	\begin{enumerate}
		\item All halfspaces in $Y$ are one-ended.
		\item $G$ acts cocompactly on $Y$.
		\item There exists a $G$-equivariant quasi-isometry $\theta:X\to Y$.
		\item The $G$-stabilizers of hyperplanes in $Y$ are subgroups of the $G$-stabilizers of hyperplanes in $X$.
	\end{enumerate}		
\end{thm:halfspaces}

We will deduce Theorem \ref{thm:halfspaces} from a repeated application of the following theorem.
Recall that an action of a group $G$ on a CAT(0) cube complex $X$ is \emph{without inversions in hyperplanes} if there is no $g\in G$ and halfspace $\h\in\cH(X)$ with $g\h=\h^*$.

\begin{thm}\label{thm:smallerstabs}
	Let $G$ be a group acting cocompactly on a one-ended locally finite essential CAT(0) cube complex $X$ without inversions in hyperplanes.
	Suppose there is a subgroup $\Gamma<G$ whose induced action on $X$ is proper and cocompact.
	If $X$ contains a halfspace with more than one end then there is a locally finite essential CAT(0) cube complex $Y$ with the following properties:
	\begin{enumerate}
		\item\label{item:GGammaact} $G$ acts cocompactly on $Y$ without inversions in hyperplanes, and the induced action of $\Gamma$ on $Y$ is proper and cocompact.
		\item There exists a $G$-equivariant quasi-isometry $\theta:X\to Y$.
		\item\label{item:trees} For each hyperplane $\hh\in\hH(X)$ there is a tree $T_{\hh}$, and there is an action of $G$ on $\sqcup_{\hh} T_{\hh}$ that is compatible with the action on $\hH(X)$. Furthermore, there is an injective $G$-equivariant map $$\sigma:\hH(Y)\to\sqcup_{\hh\in\hH(X)} VT_{\hh},$$		
		 such that:
		\begin{enumerate}[label={(\alph*)}]
			\item\label{item:finstab} $\Gamma$ acts on $\sqcup_{\hh}T_{\hh}$ with finite edge stabilizers.
			\item\label{item:cocompact} $\Gamma$ acts on $\sqcup_{\hh}T_{\hh}$ cocompactly.
			\item\label{item:nofix} There exists $\h_0\in\cH(X)$ with more than one end such that the action of $\Gamma_{\hh_0}$ on $T_{\hh_0}$ has no fixed point. 
		\end{enumerate}
	\end{enumerate}	
\end{thm}

Let's first see how to deduce Theorem \ref{thm:halfspaces}. 

\begin{proof}[Proof of Theorem \ref{thm:halfspaces}]
	We may assume that $X$ is essential by Proposition \ref{prop:essential}, and we may assume that $G$ acts on $X$ without inversions in hyperplanes by passing to the first cubical subdivision.
	If all halfspaces in $X$ are one-ended then we can take $Y=X$, otherwise apply Theorem \ref{thm:smallerstabs}. For each hyperplane $\hh\in \hH(X)$, the action of $\Gamma_{\hh}$ on $T_{\hh}$ is a finite splitting of $\Gamma_{\hh}$ over finite subgroups by \ref{item:finstab} and \ref{item:cocompact} (subdivide $T_{\hh}$ if there are edge inversions); and the $\Gamma$-stabilizers of hyperplanes in $\sigma^{-1}(T_{\hh})$ are distinct vertex groups in this splitting. Moreover, it follows from \ref{item:nofix} that this splitting is non-trivial for some hyperplane stabilizer $\Gamma_{\hh}$.
	
	If $Y$ also contains a halfspace with more than one end, then we may apply Theorem \ref{thm:smallerstabs} again to $Y$, and the hyperplane stabilizers for the new cube complex will be vertex groups in splittings of the hyperplane stabilizers for $Y$. And we can keep applying Theorem \ref{thm:smallerstabs} repeatedly, unless we obtain a cube complex $Y$ in which every halfspace is one-ended (there will be no bounded halfspaces since Theorem \ref{thm:smallerstabs} always produces an essential cube complex).
	Each hyperplane stabilizer $\Gamma_{\hh}$ for $X$ acts properly and cocompactly on $\hh$, so in particular $\Gamma_{\hh}$ is finitely presented; it then follows from Theorem \ref{thm:terminate} that the process of repeatedly applying Theorem \ref{thm:smallerstabs} must terminate after a finite number of steps. 
	The cube complex obtained at the final step is the desired cube complex $Y$ in Theorem \ref{thm:halfspaces}.
	Note that properties \ref{item:Gcocompact}--\ref{item:halfspacestabs} in Theorem \ref{thm:halfspaces} are satisfied because they hold for every application of Theorem \ref{thm:smallerstabs}.
\end{proof}

We will spend the rest of this section proving Theorem \ref{thm:smallerstabs}.
We will write $\cH=\cH(X)$ for the set of halfspaces of $X$.

\subsection{Chopping up the halfspace $\h_0$}

Let $\h_0\in\cH$ be a halfspace with more than one end, and let $\fc_0\subset\h_0$ be a finite set that separates $\h_0$ into multiple unbounded components. Passing to the convex hull, we may assume that $\fc_0$ is convex (the convex hull is finite by Lemma \ref{lem:finhull}). Also note that $\fc_0$ intersects the cubical neighborhood $N(\hh_0)$, else $\fc_0$ would separate $X$ into multiple unbounded components, contradicting one-endedness of $X$.

\begin{lem}\label{lem:boundaryunbounded}
If $\fa$ is an unbounded component of $\h_0-\fc_0$ then $\fa\cap N(\hh_0)$ is unbounded.
\end{lem}
\begin{proof}
	If not, then $\fc_0\cup(\fa\cap N(\hh_0))$ is finite and $\fa-N(\hh_0)$ is a finite union of unbounded components of $X^0-\fc_0\cup(\fa\cap N(\hh_0))$.
	The halfspace $\h^*_0$ is contained in another component of $X^0-\fc_0\cup(\fa\cap N(\hh_0))$, and is itself unbounded because $X$ is essential. This contradicts $X$ being one-ended.
\end{proof}

We now consider $G_{\h_0}$-translates of $\fc_0$, which also separate $\h_0$ into multiple unbounded components.
For convenience we will write $G_0$ in place of $G_{\h_0}=G_{\hh_0}$.
Define a wallspace $(\h_0,\cP_0)$, where $\cP_0$ consists of the components of $\h_0-g\fc_0$ for $g\in G_0$, and their complements in $\h_0$.

\begin{lem}\label{lem:finhyp}
	$\cP_0$ has finite width, and the cubing $C(\cP_0)$ has compact hyperplanes.
\end{lem}
\begin{proof}
	As in Proposition \ref{prop:cubing}\ref{item:halfspace}, halfspaces in $C(\cP_0)$ correspond to elements of $\cP_0$, and intersecting hyperplanes in $C(\cP_0)$ come from transverse elements in $\cP_0$; so both assertions of the lemma follow if there is a bound on the number of elements of $\cP_0$ transverse to any given element of $\cP_0$.	
	Local finiteness of $\h_0$ implies that $\h_0-\fc_0$ has finitely many components, so there are finitely many $G_0$-orbits in $\cP_0$. Thus it suffices to consider $\fa\in\cP_0$ a component of $\h_0-\fc_0$, and show that it is transverse to finitely many elements of $\cP_0$. If $\fb$ is a component of $\h_0-g\fc_0$ with $\fc_0,g\fc_0$ disjoint, then $\fa$ is nested with either $\fb$ or $\fb^*$, so $\fa,\fb$ are not transverse. 
	But $\fc_0$ is finite and $X$ is locally finite, so there are only finitely many sets $g\fc_0$ ($g\in G_0$) with $\fc_0\cap g\fc_0\neq\emptyset$, hence only finitely many elements of $\cP_0$ are transverse to $\fa$, as required.
\end{proof}

\begin{lem}\label{lem:finorbits}
	$\Gamma_0:=\Gamma\cap G_0$ acts on $C(\cP_0)$ with finitely many orbits of hyperplanes and finitely many orbits of cubes.
\end{lem}
\begin{proof}
	The action of $\Gamma_0$ on $\hh_0$ is cocompact, so there are finitely many $\Gamma_0$-orbits of the sets $g\fc_0$, hence finitely many $\Gamma_0$-orbits in $\cP_0$. It then follows from Proposition \ref{prop:cubing}\ref{item:halfspace} that $C(\cP_0)$ has finitely many $\Gamma_0$-orbits of hyperplanes. Every cube of $C(\cP_0)$ is contained in the cubical neighborhood of a hyperplane, and these cubical neighborhoods are finite by Lemma \ref{lem:finhyp}, hence $C(\cP_0)$ has finitely many $\Gamma_0$-orbits of cubes.
\end{proof}

\begin{lem}\label{lem:finstab}	
	The $\Gamma_0$-stabilizer of any pair of cubes in $C(\cP_0)$ is finite.
\end{lem}
\begin{proof}
	If $\fa\in\cP_0$ is a component of $\h_0-g\fc_0$, then the $\Gamma_0$-stabilizer of $\fa$ is contained in the $\Gamma_0$-stabilizer of the finite set of edges that join $\fa$ to $g\fc_0$, and this stabilizer is finite since $\Gamma$ acts properly on $X$. It follows from Proposition \ref{prop:cubing}\ref{item:halfspace} that $\Gamma_0$ has finite hyperplane stabilizers in $C(\cP_0)$. The lemma then follows because the $\Gamma_0$-stabilizer of a pair of cubes in $C(\cP_0)$ stabilizes the finite set of hyperplanes that intersect or separate them.
\end{proof}

\begin{lem}\label{lem:nofix}
	The action of $\Gamma_0$ on $C(\cP_0)$ has no fixed cube or pair of cubes.
\end{lem}
\begin{proof}
It suffices to find $\fa\in\cP_0$ and $g\in \Gamma_0$ with $g\fa\subsetneq\fa$ ($g$ \emph{skewers} the hyperplane corresponding to $\fa$ in the language of \cite{CapraceSageev11}) as then $g$ will have no fixed vertex or pair of vertices.
Indeed if $g$ fixes a vertex or pair of vertices, then $g^2$ fixes some vertex $\omega\in C(\cP_0)$, but then either $\fa\in\omega$ and $\fa\supsetneq g^2\fa\supsetneq g^4\fa\supsetneq...$ is a strictly descending infinite chain in $\omega$, or $\fa^*\in\omega$ and $\fa^*\supsetneq g^{-2}\fa^*\supsetneq g^{-4}\fa^*\supsetneq...$ is a strictly descending infinite chain in $\omega$ -- so either way we contradict $\omega$ being a DCC ultrafilter. 
	
As $\Gamma_0$ acts cocompactly on the hyperplane $\hh_0$ and its cubical neighborhood $N(\hh_0)$, it follows from Lemma \ref{lem:boundaryunbounded} that there exist $g_1,g_2\in \Gamma_0$ with $g_1\fc_0,g_2\fc_0$ contained in distinct components $\fa_1,\fa_2$ of $\h_0-\fc_0$ respectively. If $g_1\fa_1\subsetneq\fa_1$ or $g_2\fa_2\subsetneq\fa_2$ then we are done, otherwise $g_1\fa_1,g_2\fa_2$ both contain $\fc_0$. But in that case $g_2\fa_1$ is a component of $\h_0 - g_2\fc_0$ that doesn't contain $\fc_0$, so $g_2\fa_1\subsetneq g_1\fa_1$, and we are again done because $g_1^{-1}g_2\fa_1\subsetneq\fa_1$.
\end{proof}

\subsection{The trees $T_{\hh}$}\label{subsec:trees}

The action of $G_0$ on the cubical subdivision $\dot{C}(\cP_0)$ of $C(\cP_0)$ is without inversions in hyperplanes, and the hyperplanes of $\dot{C}(\cP_0)$ are finite since the hyperplanes of $C(\cP_0)$ are.
Hence, we can repeatedly apply the panel collapse procedure of Hagen--Touikan \cite[Theorem A]{HagenTouikan19} to $\dot{C}(\cP_0)$ to obtain an action of $G_0$ on a tree $T_0$. Moreover, there is a $G_0$-equivariant bijection between the vertex sets of $T_0$ and $\dot{C}(\cP_0)$ (this is not stated explicitly in \cite{HagenTouikan19} but it follows from their construction).
Thus, there is a $G_0$-equivariant bijection between $VT_0$ and the set of cubes of $C(\cP_0)$.
It follows from Lemmas \ref{lem:finstab} and \ref{lem:nofix} that $\Gamma_0$ acts on $T_0$ with finite edge stabilizers and no fixed point. 
As $\Gamma_0$ is finitely generated, there is a $\Gamma_0$-invariant subtree $T'_0\subseteq T_0$ with finitely many $\Gamma_0$-orbits of edges. There is a $\Gamma_0$-equivariant bijection between the vertices in $T_0-T'_0$ and the edges in $T_0-T'_0$, where each vertex maps to the incident edge that points towards $T'_0$, so we conclude from Lemma \ref{lem:finorbits} that $\Gamma_0$ acts cocompactly on $T_0$.

If $\{g_i\mid i\in\Omega\}$ is a left transversal of $G_0$ in $G$ then we get an induced action of $G$ on the product
$T_0\times G/G_0$, explicitly this is given by
\begin{equation}\label{inducedaction}
g\cdot(v,g_iG_0):=(g_0v,g_jG_0),
\end{equation}
where $gg_i=g_jg_0$ with $i,j\in\Omega$ and $g_0\in G_0$, and $g_0v$ refers to the action of $G_0$ on $T_0$. 
(This construction is essentially the same as the notion of induced representation from representation theory.)
We may assume that the transversal $\{g_i\}$ includes the identity element, in which case the action of $G_0$ on $T_0\times\{G_0\}$ recovers the original action of $G_0$ on $T_0$.

We can then define the trees $T_{\hh}$ from Theorem \ref{thm:smallerstabs}, and the action of $G$ on $\sqcup_{\hh}T_{\hh}$, by putting 
$$T_{g_i\hh}:=T_0\times\{g_i G_0\},$$
and letting $T_{\hh}$ be a single point for hyperplanes $\hh\notin G\cdot\hh_0$. Properties \ref{item:finstab}--\ref{item:nofix} from Theorem \ref{thm:smallerstabs} hold because there are finitely many $\Gamma$-orbits of hyperplanes in $X$, and $\Gamma_0$ acts on $T_0$ cocompactly, with finite edge stabilizers, and with no fixed point.

\subsection{The pocset $\cP$}

Let $R=\diam(\fc_0)+1$.
We know that all $G_0$-translates of $\fc_0$ lie in the $R$-neighborhood of the halfspace $\h_0^*$, so as $X$ is essential we deduce that
$$\h'_0:=\h_0-\bigcup_{g\in G_0}g\fc_0$$
is non-empty.
Define an equivalence relation $\sim$ on $\h'_0$ where $x\sim y$ if $x$ and $y$ are not separated in $\h_0$ by any set $g\fc_0$ with $g\in G_0$. Let $[x]$ denote the equivalence class of $x$, and let $\cM_0$ denote the set of equivalence classes. The equivalence relation is preserved by $G_0$, so $G_0$ acts on $\cM_0$.

\begin{figure}[H]
	\centering
	\scalebox{.8}{
		\begin{tikzpicture}[auto,node distance=2cm,
			thick,every node/.style={},
			every loop/.style={min distance=2cm},
			hull/.style={draw=none},
			]
			\tikzstyle{label}=[draw=none,font=\Large]
			
			\draw[rounded corners=10pt] (-.5,4)--(-.5,1)--(1.5,1)--(1.5,4);
			\draw[rounded corners=10pt] (-4.5,4)--(-4.5,1)--(-2.5,1)--(-2.5,4);
			\draw[rounded corners=10pt] (3.5,4)--(3.5,1)--(5.5,1)--(5.5,4);
			
			\draw[rounded corners=10pt] (-.5,-3)--(-.5,-1)--(2.5,-1)--(2.5,-3);
			\draw[rounded corners=10pt] (-6,-1)--(-4,-1)--(-4,-3);
			\draw[rounded corners=10pt] (4,-3)--(4,-1)--(6,-1);
			
			\draw[fill=Green,opacity=.2]{[rounded corners=10](-2.5,4)--(-2.5,1)}--(-3,1)--(-3,0)--(0,0)--(0,1)
			{[rounded corners=10]--(-.5,1)}--(-.5,4)--(-2.5,4);
			
			\draw[blue,line width=1pt] (0,0) rectangle (1,1);
			\draw[fill=blue,opacity=.5] (0,0) rectangle (1,1);
			\draw[blue,line width=1pt] (-4,0) rectangle (-3,1);
			\draw[fill=blue,opacity=.5] (-4,0) rectangle (-3,1);
			\draw[blue,line width=1pt] (4,0) rectangle (5,1);
			\draw[fill=blue,opacity=.5] (4,0) rectangle (5,1);
			
			\draw[red,line width=1pt] (-6,0)--(6,0);
			\draw[draw=red,fill=black,-triangle 90, ultra thick](2,0)--(2,1);
			
			\node[circle,draw,fill,inner sep=0pt,minimum size=5pt] at (-2,2){};
			
			\node[label] at (-1.6,2){$x$};
			\node[label,blue] at (-3.5,1.4){$g_1\fc_0$};
			\node[label,blue] at (.5,1.4){$g_2\fc_0$};
			\node[label,blue] at (4.5,1.4){$g_3\fc_0$};
			\node[label,red] at (2.6,.5){$\h_0$};
			\node[label,Green] at (-1.5,3.5){$[x]$};
				
		\end{tikzpicture}
	}
	\caption{Cartoon of an equivalence class $[x]\in\cM_0$. The elements $g_1,g_2,g_3$ are in $G_0$.}\label{fig:[x]}
\end{figure}
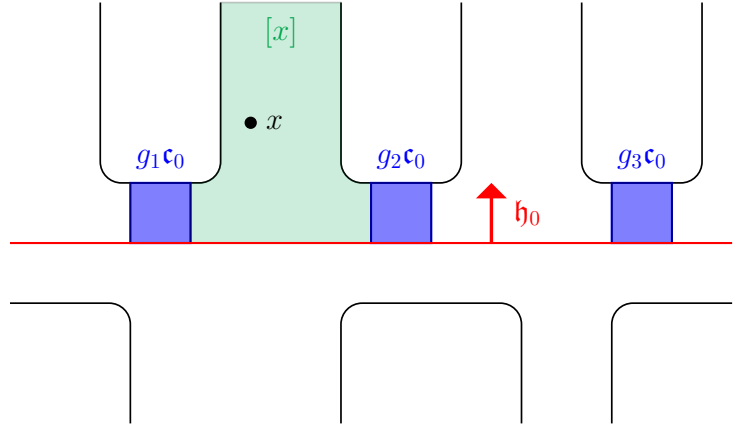

Now define $\cP$ to be the set of all pairs $(\fa,\h)$, with $\fa\subseteq X^0$ and $\h\in\cH$, that arise in one of the following three ways:
\begin{equation}
	(\fa,\h)=
	\begin{cases}
		(\h,\h),              & \h\notin G\cdot\{\h_0,\h_0^*\},\\
		(g[x],g\h_0),&g\in G,\,[x]\in\cM_0,\\
			(g[x]^*,g\h_0^*),&g\in G,\,[x]\in\cM_0,	
	\end{cases}
\end{equation}
where $[x]^*$ is the complement of $[x]$ in $X^0$ -- in fact in this section we will always denote the complement of $\fa\subseteq X^0$ by $\fa^*$.
Define an action of $G$ on $\cP$ by $g\cdot(\fa,\h):=(g\fa,g\h)$.
Also define an involution on $\cP$ by $(\fa,\h)\mapsto(\fa,\h)^*:=(\fa^*,\h^*)$.
Finally, we make $\cP$ into a pocset with the ordering $(\fa_1,\h_1)\leq(\fa_2,\h_2)$ if $\fa_1\subsetneq\fa_2$ or $(\fa_1,\h_1)=(\fa_2,\h_2)$.

Note that $\cP$ looks very much like a wallspace on $X^0$ if one just considers the first coordinate of each pair $(\fa,\h)\in\cP$, but the second coordinate will be needed in order to relate $\cP$ and its cubing to $X$.

\subsection{The map $\sigma$}\label{subsec:sigma}

Each $x\in\h'_0$ defines a DCC ultrafilter on $\cP_0$ given by
$$\lambda_x:=\{\fa\in\cP_0\mid x\in\fa\}.$$
Let $x\sim y$. For any $g\in G_0$ we know that $x$ and $y$ lie in the same component of $\h_0-g\fc_0$, so it follows that $\lambda_x=\lambda_y$. Conversely, if $x\nsim y$ then there exists $g\in G_0$ such that $x$ and $y$ are separated by $g\fc_0$, so there exists a component $\fa$ of $\h_0-g\fc_0$ containing $x$ but not $y$, and it follows that $\lambda_x\neq\lambda_y$. This yields the following lemma.

\begin{lem}
	There is a $G_0$-equivariant injection $\cM_0\to C(\cP_0)^0=VT_0$ given by $[x]\mapsto\lambda_x$.
\end{lem}

Recall from Section \ref{subsec:trees} that the action of $G_0$ on $T_0$ extends to an action of $G$ on $\sqcup_{\hh\in\hH(X)} VT_{\hh}$, with $T_0=T_{\hh_0}$.
Also recall that $T_{\hh}$ consists of a single vertex if $\h\notin G\cdot\{\h_0,\h_0^*\}$.
We can then extend the map $\cM_0\to VT_0$ above to a $G$-equivariant map
$$\sigma:\cP\to\bigsqcup_{\hh\in\hH(X)} VT_{\hh},$$	
by setting
\begin{align*}
	&\sigma(\h,\h)=T_{\hh}, &&\h\notin G\cdot\{\h_0,\h_0^*\},&&&&\\
	&\sigma(g[x],g\h_0)=g\lambda_x, &&g\in G,\,[x]\in\cM_0,&&&&\\
	&\sigma(g[x]^*,g\h_0^*)=g\lambda_x, &&g\in G,\,[x]\in\cM_0.&&&&\\
\end{align*}
We clearly have $\sigma(\fa,\h)=\sigma(\fa^*,\h^*)$ for all $(\fa,\h)\in\cP$, and this is the only failure of injectivity since $[x]\mapsto\lambda_x$ is injective. Hence $\sigma$ descends to an injective $G$-equivariant map
$$\sigma:\hat{\cP}\to\bigsqcup_{\hh\in\hH(X)} VT_{\hh}.$$
We will soon construct $Y$ as the cubing of $\cP$, and then $\hat{\cP}$ will be identified with $\hH(Y)$ by Proposition \ref{prop:cubing}\ref{item:halfspace}, thus making $\sigma$ the required map in Theorem \ref{thm:smallerstabs}\ref{item:trees}.

\subsection{The cubing $Y$}

We will define the cube complex $Y$ to be the cubing of the pocset $(\cP,\leq)$. So we must show that there exist DCC ultrafilters on $\cP$ and that $\cP$ has finite width.

For $x\in X^0$ let
$$\theta(x):=\{(\fa,\h)\in\cP\mid x\in\fa\}.$$
This is clearly an ultrafilter on $\cP$, and we will show over the next four lemmas that it is DCC.
As $X$ is locally finite and cocompact, there exists a function $k:\mathbb{N}\to\mathbb{N}$ such that the $r$-neighborhood of any vertex or edge in $X$ intersects at most $k(r)$ cubical neighborhoods of hyperplanes.
We will use this function as a source of constants throughout this section.

\begin{lem}\label{lem:eleavefa}
	If $(\fa,\h)\in\cP$ and $e$ is an edge that joins a vertex in $\fa$ to a vertex in $\fa^*$, then 
	$$d(e,\h),d(e,\h^*)\leq R.$$
	Moreover, for each edge $e$ in $X$ there are at most $2k(R)$ elements $(\fa,\h)\in\cP$ such that $e$ joins a vertex in $\fa$ to a vertex in $\fa^*$.
\end{lem}
\begin{proof}
	Let $(\fa,\h)\in\cP$ with $e$ joining a vertex in $\fa$ to a vertex in $\fa^*$.
	We have two cases:
	\begin{enumerate}
		\item\label{fa=h} If $\h\notin G\cdot\{\h_0,\h_0^*\}$ then $\fa=\h$ and $d(e,\h)=d(e,\h^*)=0$.
		\item\label{fah=gzgh0} If $\h\in G\cdot\{\h_0,\h_0^*\}$, then by symmetry we may assume that $(\fa,\h)=([x],\h_0)$ for $[x]\in\cM_0$. Observe that $e$ leaves $[x]$ and either crosses the hyperplane $\hh=\hh_0$ or enters a set $g\fc_0$ with $g\in G_0$. Since the sets $g\fc_0$ have diameter at most $R-1$, we deduce that $d(e,\h),d(e,\h^*)\leq R$.
	\end{enumerate}
	
	We now turn to the second part of the lemma.
	By definition of $k(R)$, for each $e$ there are at most $k(R)$ possibilities for the hyperplane $\hh$, so at most $2k(R)$ possibilities for the halfspace $\h$. And for each $\h$ we claim that there is at most one possibility for $\fa$. 
	In case \ref{fa=h} $\fa=\h$ is determined.
	In case \ref{fah=gzgh0} the edge $e$ leaves $[x]$, and its other endpoint is not in $\h'_0$, hence not in any other $[x']\in\cM_0$, so $[x]$ is uniquely determined.
\end{proof}

\begin{lem}\label{lem:xyfarfromaa*}
If $x,y\in X^0$, $\h\in\cH$ and $S\geq0$ is an integer with $d(x,\h^*),d(y,\h)>R+S$, then there exists $(\fa,\h)\in\cP$ with $d(x,\fa^*),d(y,\fa)>S$.
\end{lem}
\begin{proof}
	If $\h\notin G\cdot\{\h_0,\h_0^*\}$ then we can put $\fa=\h$.
	If $\h\in G\cdot\{\h_0,\h_0^*\}$ then by symmetry we may assume that $\h=\h_0$. Then $d(x,\h_0^*)>R+S$, so $x\in\h'_0$ and we can put $\fa=[x]$.  We must have $d(x,\fa^*)>S$ since $d(x,\h_0^*)>R+S$ and any edge $e$ joining $\fa$ to $\fa^*$ has $d(e,\h_0^*)\leq R$ by Lemma \ref{lem:eleavefa}. 
	On the other hand we have $y\in\h_0^*\subseteq \fa^*$, so $d(y,\fa)\geq d(y,\h_0)>R+S>S$.
\end{proof}

\begin{figure}[H]
	\centering
	\scalebox{.8}{
		\begin{tikzpicture}[auto,node distance=2cm,
			thick,every node/.style={},
			every loop/.style={min distance=2cm},
			hull/.style={draw=none},
			]
			\tikzstyle{label}=[draw=none,font=\Large]
			
			\draw[rounded corners=10pt] (-.5,4)--(-.5,1)--(1.5,1)--(1.5,4);
			\draw[rounded corners=10pt] (-5.5,4)--(-5.5,1)--(-3.5,1)--(-3.5,4);
			
			\draw[rounded corners=10pt] (-.5,-3)--(-.5,-1)--(2.5,-1)--(2.5,-3);
			\draw[rounded corners=10pt] (-6,-1)--(-4,-1)--(-4,-3);
			
			\draw[fill=Green,opacity=.2]{[rounded corners=10](-3.5,4)--(-3.5,1)}--(-4,1)--(-4,0)--(0,0)--(0,1)
			{[rounded corners=10]--(-.5,1)}--(-.5,4)--(-3.5,4);
			
			\draw[blue,line width=1pt] (0,0) rectangle (1,1);
			\draw[fill=blue,opacity=.5] (0,0) rectangle (1,1);
			\draw[blue,line width=1pt] (-5,0) rectangle (-4,1);
			\draw[fill=blue,opacity=.5] (-5,0) rectangle (-4,1);
			
			\draw[red,line width=1pt] (-6,0)--(3,0);
			\draw[draw=red,fill=black,-triangle 90, ultra thick](2,0)--(2,1);
			
			\node[circle,draw,fill,inner sep=0pt,minimum size=5pt] at (-2,2){};
			\node[circle,draw,fill,inner sep=0pt,minimum size=5pt] at (-3,-2){};
			
			\node[label,blue] at (-4.5,1.4){$g_1\fc_0$};
			\node[label,blue] at (.5,1.4){$g_2\fc_0$};
			\node[label] at (-1.6,2){$x$};
			\node[label] at (-2.6,-2){$y$};
			\node[label,red] at (2.6,.5){$\h_0$};
			\node[label,Green] at (-2,3.5){$\fa=[x]$};
			
			\draw  (-2, 0) edge[<->] node[right] {$>R+S$} (-2, 1.9);
			\draw  (-3, 0) edge[<->] node[right] {$>R+S$} (-3, -1.9);
			\draw  (-3.8, 0) edge[<->] node[right] {$\leq R$} (-3.8, 1);
			
		\end{tikzpicture}
	}
	\caption{Cartoon picture for the proof of Lemma \ref{lem:xyfarfromaa*}. The elements $g_1,g_2$ are in $G_0$.}\label{fig:xyfar}
\end{figure}
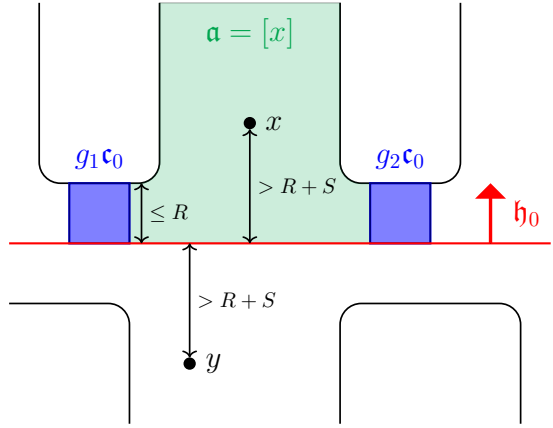

\begin{lem}\label{lem:thetaineq}
	$d(x,y)-2k(R)\leq|\theta(x)-\theta(y)|\leq k(R)d(x,y)$ for all $x,y\in X^0$.
\end{lem}
\begin{proof}
	Fix $x,y\in X^0$ and $\gamma$ a geodesic between them (in $X^1$).	
	Let $e$ be an edge on $\gamma$, and suppose it is dual to a hyperplane $\hh$ with $x\in\h$ and $y\in\h^*$.
	Further suppose that the cubical neighborhood $N(\hh)$ does not intersect the $R$-neighborhoods of $x$ and $y$, so $d(x,\h^*),d(y,\h)>R$.
	By Lemma \ref{lem:xyfarfromaa*} there exists $(\fa,\h)\in\cP$ with $x\in\fa$ and $y\in\fa^*$, so $(\fa,\h)\in\theta(x)-\theta(y)$.
	By definition of $k(R)$, there are at least $d(x,y)-2k(R)$ such edges $e$, and they are dual to distinct hyperplanes since $\gamma$ is a geodesic (Lemma \ref{lem:d}), so this proves the first inequality.
	
	For the second inequality, for each $(\fa,\h)\in\theta(x)\triangle\theta(y)$ there is an edge $e$ on $\gamma$ joining a vertex in $\fa$ to a vertex in $\fa^*$. By Lemma \ref{lem:eleavefa} we have $|\theta(x)\triangle\theta(y)|\leq 2k(R)d(x,y)$, so $|\theta(x)-\theta(y)|\leq k(R)d(x,y)$.	
\end{proof}

\begin{lem}
	$\theta(x)$ is a DCC ultrafilter for all $x\in X^0$.
\end{lem}
\begin{proof}
	Let $(\fa,\h)\in\theta(x)$. Pick $y\in\fa^*$. Any $(\fa',\h')\in\theta(x)$ with $(\fa',\h')\leq(\fa,\h)$ satisfies $\fa'\subseteq\fa$, so $(\fa',\h')\in\theta(x)-\theta(y)$. We deduce from Lemma \ref{lem:thetaineq} that no strictly descending infinite chain in $\theta(x)$ contains $(\fa,\h)$. But $(\fa,\h)$ was an arbitrary element of $\theta(x)$, hence $\theta(x)$ is DCC.
\end{proof}

We now prove over the next three lemmas that $\cP$ has finite width. This involves the convex hull $\Hull(\fa)$ and the metric $E$-neighborhood $\cC_E(\fa)$ of a set $\fa\subset X^0$ (see Section \ref{subsec:CCC}).

\begin{lem}\label{lem:hull}
	There exists an integer $E>0$ such that $\Hull(\fa)\subseteq\cN_E(\fa)$ for all $(\fa,\h)\in\cP$.
\end{lem}
\begin{proof}
	If $\h\notin G\cdot\{\h_0,\h_0^*\}$ then $\fa=\h$ is already convex, so it suffices to find $E>0$ such that $\Hull([x])\subseteq\cN_E([x])$ and $\Hull([x]^*)\subseteq\cN_E([x]^*)$ for all $[x]\in\cM_0$.
	
	Fix $[x]\in\cM_0$. As in Section \ref{subsec:sigma}, $[x]$ defines a vertex $\lambda_x\in C(\cP_0)^0$.
	Let $\{e_\tau\}$ be the set of edges incident to $\lambda_x$.
	Lemma \ref{lem:finhyp} tells us that the hyperplanes of $C(\cP_0)$ are compact, and Lemma \ref{lem:finorbits} tells us that $C(\cP_0)$ has finitely many $G_0$-orbits of hyperplanes, hence $C(\cP_0)$ has finitely many $G_0$-orbits of edges.
	It follows that the $G_0$-stabilizer of $[x]$ has finitely many orbits in $\{e_\tau\}$.
	Each edge $e_\tau$ leaves a halfspace of $C(\cP_0)$ containing $\lambda_x$, and such a halfspace corresponds to a component $\fa_\tau$ of $\h_0-g_\tau\fc_0$ containing $[x]$, for some $g_\tau\in G_0$.
	Moreover, any $y\in\h'_0-[x]$ defines a vertex $\lambda_y\neq\lambda_x$ in $C(\cP_0)$, so $\lambda_y$ is separated from $\lambda_x$ by a hyperplane dual to some $e_\tau$, and therefore $y$ is separated from $[x]$ by one of the sets $g_\tau\fc_0$.
	
	As the $G_0$-stabilizer of $[x]$ has finitely many orbits in $\{e_\tau\}$, we see that there is an integer $E_1>0$ such that all the sets $g_\tau\fc_0$ are contained in the $E_1$-neighborhood of $[x]$.
	$G_0$ acts cocompactly on $\h_0-\h'_0$, so there is also an integer $E_2>0$ such that $\h_0$ is contained in the $E_2$-neighborhood of $\h'_0$.
	Let $E:=E_1+2E_2$.
	We now claim that every $y\in\h_0-\cN_E([x])$ is separated from $[x]$ by one of the sets $g_\tau\fc_0$.
	Indeed, given such a $y$ there exists $y'\in\h'_0$ with $d(y,y')\leq E_2$, and $y'\notin\cN_{E_1+E_2}([x])$, so in particular $y'\notin[x]$, and there exists $g_\tau\fc_0$ separating $y'$ from $[x]$. But $g_\tau\fc_0\subseteq\cN_{E_1}([x])$, so $d(y',g_\tau\fc_0)>E_2$, hence $y$ is also separated from $[x]$ by $g_\tau\fc_0$.
	
	If $y\in\h_0$ is separated from $[x]$ by $g_\tau\fc_0$, then in particular $y\notin g_\tau\fc_0$. As $g_\tau\fc_0$ is convex, there exists a halfspace $\h\in\cH$ with $g_\tau\fc_0\subseteq\h$ and $y\in\h^*$ (Lemma \ref{lem:d}). The intersection $\h_0\cap\h^*$ is convex (Lemma \ref{lem:intersect}), so in particular connected, thus we must have $[x]\subseteq\h$ (else there is a path in $\h_0\cap\h^*$ joining $y$ to $[x]$ that avoids $g_\tau\fc_0$).
	
	$\Hull([x])$ is the intersection of halfspaces containing $[x]$, so our arguments so far imply that any $y\in\h_0-\cN_E([x])$ lies outside $\Hull([x])$. We also know that $[x]\subseteq\h_0$, so $\Hull([x])\subseteq\h_0$. This proves that $\Hull([x])\subseteq\cN_E([x])$.
	
	Finally we turn to $\Hull([x]^*)$. 
	Enlarging $E$ so that $E\geq R\dim X$, we claim that 
	$$\Hull([x]^*)\subseteq\cN_E([x]^*).$$
	Suppose $y\in[x]-\cN_E([x]^*)$. Our task is to find a halfspace containing $[x]^*$ but not $y$.
	Observe from Remark \ref{remk:thickconvex} that $(\h_0^*)^{+R}$ is convex, and from Remark \ref{remk:thickneigh} that $$\cN_R(\h_0^*)\subseteq(\h_0^*)^{+R}\subseteq\cN_{R\dim X}(\h_0^*).$$
	We know that $\h_0^*\subseteq[x]^*$, so $y\notin\cN_{R\dim X}(\h_0^*)$ and $y\notin(\h_0^*)^{+R}$, so by convexity of $(\h_0^*)^{+R}$ there is a halfspace $(\h_0^*)^{+R}\subseteq\h\in\cH$ with $y\in\h^*$. The halfspace $\h^*$ is disjoint from $\cN_R(\h_0^*)$, so lies in $\h'_0$, and $\h^*$ is connected so it lies inside one of the classes in $\cM_0$. As $y\in\h^*\cap[x]$ we deduce that $\h^*\subseteq[x]$, so $[x]^*\subseteq\h$ as required. 
\end{proof}

\begin{lem}\label{lem:twohulls}
	For $(\fa,\h)\in\cP$ the intersection $\Hull(\fa)^{+1}\cap\Hull(\fa^*)^{+1}$ is non-empty and is contained in the $(R+E+\dim X)$-neighborhood of $N(\hh)$.
\end{lem}
\begin{proof}
The intersection $\Hull(\fa)^{+1}\cap\Hull(\fa^*)^{+1}$ is non-empty because it contains any edge that joins $\fa$ to $\fa^*$.
Lemma \ref{lem:hull} says that $\Hull(\fa)\subseteq\cN_E(\fa)$, so $\Hull(\fa)^{+1}\subseteq\cN_{E+\dim X}(\fa)$ by Remark \ref{remk:thickneigh}. Similarly $\Hull(\fa^*)^{+1}\subseteq\cN_{E+\dim X}(\fa^*)$.
Thus any vertex in $\Hull(\fa)^{+1}\cap\Hull(\fa^*)^{+1}$ is within distance $E+\dim X$ of an edge that joins $\fa$ to $\fa^*$, so it lies in the $(R+E+\dim X)$-neighborhood of $N(\hh)$ by Lemma \ref{lem:eleavefa}.
\end{proof}

\begin{lem}
	$\cP$ has finite width.
\end{lem}
\begin{proof}
	Let $\{(\fa_i,\h_i)\}$ be a finite collection of pairwise transverse elements of $\cP$. It follows easily from the construction of $\cP$ that the $\h_i$ are distinct.
	Remark \ref{remk:thickconvex} implies that all the sets $\Hull(\fa_i)^{+1}$ and $\Hull(\fa^*_i)^{+1}$ are convex, and they pairwise intersect by Lemma \ref{lem:twohulls} and the fact that the $(\fa_i,\h_i)$ are pairwise transverse. So Lemma \ref{lem:intersect} tells us that
	$$\bigcap_i (\Hull(\fa_i)^{+1}\cap\Hull(\fa^*_i)^{+1})\neq\emptyset.$$
	But, by Lemma \ref{lem:twohulls}, any vertex in this intersection is in the $(R+E+\dim X)$-neighborhood of $N(\hh_i)$ for all $i$, thus the size of the collection $\{(\fa_i,\h_i)\}$ is bounded by $k(R+E+\dim X)$.
\end{proof}

As promised, we can now define $Y:=C(\cP)$. The construction also gives us a $G$-equivariant map $\theta:X^0\to Y^0$, which is a quasi-isometric embedding by Lemmas \ref{lem:domega} and \ref{lem:thetaineq}.

\subsection{The quasi-isometries $\theta$ and $\phi$}

Next we will show that $\theta: X^0\to Y^0$ is a quasi-isometry by constructing a coarse inverse $\phi$.
For $\mu\in Y$, define the subset $\omega_\mu\subseteq\cH$ to consist of all halfspaces $\h$ with $\fa\subsetneq\h$ for some $(\fa,\h')\in\mu$.

\begin{lem}\label{lem:DCCpartial}
	$\omega_\mu$ is a DCC partial ultrafilter on $\cH$.
\end{lem}
\begin{proof}
	First let's prove DCC, so suppose $\h_1\supsetneq\h_2\supsetneq\h_3\supsetneq...$ is a strictly descending infinite chain in $\omega_\mu$.
	For each $\h_i$ pick $(\fa_i,\h'_i)\in\mu$ with $\fa_i\subsetneq\h_i$.
	The intersection $\cap_i\h_i$ is empty, so each $\fa_j$ is contained in only finitely many $\h_i$, hence the sequence $(\fa_j,\h'_j)$ contains infinitely many distinct elements.
	For $x\in\h_1^*$ we have $x\notin\fa_i$ for all $i$, so $(\fa_i,\h'_i)\in\mu-\theta(x)$ for all $i$, but then $\mu-\theta(x)$ is infinite, contradicting $\mu,\theta(x)\in Y^0$.
	
	To see that $\omega_\mu$ is a partial ultrafilter, suppose that $\h_1,\h_2\in\omega_\mu$ are disjoint. Then there exist $(\fa_1,\h'_1),(\fa_2,\h'_2)\in\mu$ with $\fa_i\subsetneq\h_i$.
	In turn this implies that $\fa_1\subsetneq\fa_2^*$, so $(\fa_1,\h'_1)\leq(\fa_2^*,(\h'_2)^*)$, contradicting $\mu$ being an ultrafilter.
\end{proof}

\begin{lem}\label{lem:bomega}
	The intersection $\fb_\mu:=\bigcap_{\h\in\omega_\mu}\h$	is non-empty, and $\fb_\mu\subseteq\h\in\cH$ implies $\h\in\omega_\mu$.
\end{lem}
\begin{proof}
	$\omega_\mu$ is a DCC partial ultrafilter by Lemma \ref{lem:DCCpartial}, and it is clear that $\h_1\in\omega_\mu$ and $\h_1\subseteq\h_2\in\cH$ implies $\h_2\in\omega_\mu$. Hence we can apply Lemma \ref{lem:partial=subcomplex} to deduce that $\omega_\mu$ can be extended to a DCC complete ultrafilter, and that the set of such extensions is precisely the set $\fb_\mu$ defined above (identifying $X$ with $C(\cH)$ by Theorem \ref{thm:duality}). Thus $\fb_\mu$ is non-empty.
	
	For the second part, suppose $\fb_\mu\subseteq\h\in\cH$ with $\h\notin\omega_\mu$. Then we can pick $x\in\h^*$, and  define a DCC complete ultrafilter $\omega_x$ on $\cH$ with $\h^*\in\omega_x$. 
	We know that $\h^*\notin\omega_\mu$ as that would contradict $\fb_\mu\subseteq\h$, so we may apply Lemma \ref{lem:partial=subcomplex} with $\omega_0=\omega_x$ to extend $\omega_\mu$ to a DCC complete ultrafilter $\bar{\omega}$ with $\h^*\in\bar{\omega}$. But we said above that such an extension corresponds to a vertex of $\fb_\mu$, so this implies $\h^*\cap\fb_\mu\neq\emptyset$, once again contradicting $\fb_\mu\subseteq\h$.
\end{proof}

\begin{lem}\label{lem:N1Hull}
	$\fb_\mu$ is contained in the intersection $\bigcap_{(\fa,\h)\in\mu}\cN_1(\Hull(\fa))$.
\end{lem}
\begin{proof}
	Let $(\fa,\h)\in\mu$ and suppose that $x\in X^0-\cN_1(\Hull(\fa))$.
	Lemma \ref{lem:d} implies that there are distinct $\h_1,\h_2\in\cH$ with $\Hull(\fa)\subseteq\h_i$ and $x\in\h_i^*$. In particular, we either have $\fa\subsetneq\h_1$ or $\fa\subsetneq\h_2$, so at least one of $\h_1,\h_2$ is in $\omega_\mu$, hence $x\notin\fb_\mu$.
\end{proof}

\begin{lem}\label{lem:bmudiameter}
	The diameter of $\fb_\mu$ is at most $D:=2k(R+E+1)$.
\end{lem}
\begin{proof}
	Suppose $x,y\in\fb_\mu$ with $d(x,y)>D$.
	Then by considering the hyperplanes that separate $x$ and $y$, the definition of $k:\mathbb{N}\to\mathbb{N}$ tells us that there exists $\h\in\cH$ with $d(x,\h^*),d(y,\h)>R+E+1$.
	By Lemma \ref{lem:xyfarfromaa*}, there exists $(\fa,\h)\in\cP$ with $d(x,\fa^*),d(y,\fa)>E+1$.
	We know that $\mu$ is an ultrafilter, so one of $(\fa,\h)$ and $(\fa^*,\h^*)$ is in $\mu$ -- say $(\fa,\h)\in\mu$.
	Lemma \ref{lem:hull} tells us that $\Hull(\fa)\subseteq\cN_E(\fa)$, so $y\notin\cN_1(\Hull(\fa))$.
	But then Lemma \ref{lem:N1Hull} implies that $y\notin\fb_\mu$, a contradiction.
\end{proof}

We now have a coarsely well-defined map $\phi:Y^0\to X^0$ by picking $\phi(\mu)\in\fb_\mu$ for each $\mu\in Y^0$.
Our next task is to prove that $\phi$ is a coarse inverse to $\theta$.

\begin{lem}\label{lem:inverse=phi}
	$\sup_{\mu\in Y^0}d(\mu,\theta\phi(\mu))<\infty$
\end{lem}
\begin{proof}
	Let $\mu\in Y^0$.
	Suppose $(\fa,\h)\in\mu-\theta\phi(\mu)$. Then $\phi(\mu)\notin\fa$.
	But $\phi(\mu)\in\fb_\mu\subseteq\cN_1(\Hull(\fa))$ by Lemma \ref{lem:N1Hull}, so $\phi(\mu)\in\cN_{E+1}(\fa)$ by Lemma \ref{lem:hull}.
	We deduce that there exists an edge $e$ in the $(E+1)$-neighborhood of $\phi(\mu)$ that joins $\fa$ to $\fa^*$.
	By Lemma \ref{lem:eleavefa}, we can then bound $|\mu-\theta\phi(\mu)|$ by the product of $2k(R)$ and the number of edges in the $(E+1)$-neighborhood of $\phi(\mu)$, and this can be bounded independently of $\mu$ since $X$ is locally finite and cocompact.
\end{proof}

It follows from Lemmas \ref{lem:thetaineq} and \ref{lem:inverse=phi} that $\theta$ is a quasi-isometry with coarse inverse $\phi$.

\subsection{The actions of $G$ and $\Gamma$ on $Y$}

\begin{lem}\label{lem:Ylocfinite}
	$Y$ is locally finite.
\end{lem}
\begin{proof}
Let $\mu\in Y^0$. By Lemma \ref{lem:edges}, to show that $Y$ is locally finite at $\mu$ we must show that $\mu$ has finitely many $\leq$-minimal elements. 

Let $(\fa,\h)\in\mu$. Lemma \ref{lem:N1Hull} tells us that $\phi(\mu)\in\cN_1(\Hull(\fa))$, so $d(\phi(\mu),\fa)\leq E+1$ by Lemma \ref{lem:hull}.
Now suppose that $d(\phi(\mu),\fa^*)>D+E$.
Then Lemmas \ref{lem:hull} and \ref{lem:bmudiameter} imply that $\fb_\mu$ and $\Hull(\fa^*)$ are disjoint, and since these are convex sets Lemma \ref{lem:d} provides us with $\fb_\mu\subseteq\h\in\cH$ such that $\Hull(\fa^*)\subseteq\h^*$ -- in particular $\h\subseteq\fa$.
Lemma \ref{lem:bomega} implies that $\h\in\omega_\mu$, so there exists $(\fa',\h')\in\mu$ with $\fa'\subsetneq\h$.
Then $\fa'\subsetneq\fa$ and $(\fa',\h')\leq(\fa,\h)$, so $(\fa,\h)$ is not minimal in $\mu$.

Our argument so far implies that any minimal $(\fa,\h)\in\mu$ has $d(\phi(\mu),\fa),d(\phi(\mu),\fa^*)\leq D+E$. But for such an $(\fa,\h)$ there exists an edge $e$ in the $(D+E)$-neighborhood of $\phi(\mu)$ joining $\fa$ to $\fa^*$. It follows from Lemma \ref{lem:eleavefa} and the local finiteness of $X$ that $\mu$ has finitely many minimal elements.
\end{proof}

We now verify property \ref{item:GGammaact} of Theorem \ref{thm:smallerstabs}, which concerns the actions of $G$ and $\Gamma$ on $Y$.
We already know that $\theta:X\to Y$ is a $G$-equivariant quasi-isometry, and that $G$ acts cocompactly on $X$, so it follows from Lemma \ref{lem:Ylocfinite} that $G$ acts cocompactly on $Y$.
It is also clear that $G$ acts on $Y$ without inversions in hyperplanes: we have a pocset isomorphism $(\cH(Y),\subseteq)\cong(\cP,\leq)$, and $g(\fa,\h)=(\fa^*,\h^*)$ for $(\fa,\h)\in\cP$ and $g\in G$ implies that $g\h=\h^*$, contradicting the fact that $G$ acts on $X$ without inversions in hyperplanes.

We also know that $\theta:X\to Y$ is a $\Gamma$-equivariant quasi-isometry, and that $\Gamma$ acts cocompactly on $X$, so again it follows from Lemma \ref{lem:Ylocfinite} that $\Gamma$ acts cocompactly on $Y$. Similarly, $\Gamma$ acts properly on $Y$ because it acts properly on $X$.

Finally, to ensure that $Y$ is essential we can replace it with its essential core using Proposition \ref{prop:essential} (noting that the closest point projection from $Y$ to its essential core is a $G$-equivariant quasi-isometry, and that the set of hyperplanes of the essential core is a subset of the set of hyperplanes of $Y$).
This completes the proof of Theorem \ref{thm:smallerstabs}.

\bigskip
\section{Reducing to deep quarterspaces}\label{sec:quarterspaces}

In this section we prove the following theorem. 

\theoremstyle{plain}
\newtheorem*{thm:quarterspaces}{Theorem \ref{thm:quarterspaces}}
\begin{thm:quarterspaces}
	Let $G$ be a group acting cocompactly on a CAT(0) cube complex $X$. Then there is a CAT(0) cube complex $Y$ with the following properties:
	\begin{enumerate}
		\item All quarterspaces in $Y$ are deep.
		\item $G$ acts cocompactly on $Y$.
		\item There exists a $G$-equivariant quasi-isometry $\phi:Y\to X$.
		\item $\phi$ maps each halfspace in $Y$ to within finite Hausdorff distance of a halfspace in $X$.
		\item $Y$ is locally finite if $X$ is locally finite.
	\end{enumerate}		
\end{thm:quarterspaces}

Theorem \ref{thm:quarterspaces} follows from the following theorem by induction on the number of $G$-orbits of vertices in $X$.

\begin{thm}\label{thm:smaller}
	Let $G$ be a group acting cocompactly on a CAT(0) cube complex $X$. If $X$ contains a shallow quarterspace then there is a CAT(0) cube complex $Y$ with the following properties:
	\begin{enumerate}
		\item $G$ acts cocompactly on $Y$.
		\item $Y$ has fewer $G$-orbits of vertices than $X$.
		\item There exists a $G$-equivariant quasi-isometry $\phi:Y\to X$.
		\item $\phi$ maps each halfspace in $Y$ to within finite Hausdorff distance of a halfspace in $X$.
		\item $Y$ is locally finite if $X$ is locally finite.
	\end{enumerate}	
\end{thm}

We will spend the rest of this section proving Theorem \ref{thm:smaller}.
If $X$ is bounded then we can take $Y$ to be a single vertex and $\phi(Y)$ to be the center of $X$ \cite[Proposition II.2.7]{BridsonHaefliger99}, so henceforth we assume that $X$ is unbounded.
By Proposition \ref{prop:essential} we may assume that $X$ is essential.
Throughout this section we will write $\cH=\cH(X)$ for the set of halfspaces of $X$.

\subsection{Depth-0 quarterspaces}

We need the following lemma to characterize depth-0 quarterspaces in terms of inclusions of halfspaces.

\begin{lem}\label{lem:depth0}
	Let $\h_1\cap\h_2$ be a quarterspace. Then $\h_1\cap\h_2$ has depth 0 if and only if any halfspace $\h\subsetneq\h_1$ (resp. $\h\subsetneq\h_2$) satisfies $\h\subsetneq\h_2^*$ (resp. $\h\subsetneq\h_1^*$).
\end{lem}
\begin{proof}
Suppose that $\h_1\cap\h_2$ has depth greater than 0. Then there exists $x\in \h_1\cap\h_2$ with $d(x,\h^*_1\cap\h^*_2)\geq3$. Lemma \ref{lem:d} implies that there is a halfspace $x\in\h\in\cH-\{\h_1,\h_2\}$ with $\h\cap\h_1^*\cap\h_2^*=\emptyset$. Using Lemma \ref{lem:intersect}, we can then assume without loss of generality that $\h\cap\h_1^*=\emptyset$, which means $\h\subsetneq\h_1$. Plus we know that $x\in\h\cap\h_2$, so $\h\nsubseteq\h_2^*$.

Conversely, if there is a halfspace $\h\subsetneq\h_1$ that doesn't satisfy $\h\subsetneq\h_2^*$, then $\h\neq\h_2^*$ (as $\h_1,\h_2$ are transverse) and $\h\cap\h_2\neq\emptyset$. Therefore $\h\cap\h_1^*\cap\h_2^*=\emptyset$ and there exists $x\in\h\cap\h_1\cap\h_2$, and such $x$ satisfies $d(x,\h^*_1\cap\h^*_2)\geq3$, so that $\h_1\cap\h_2$ has depth greater than 0.
\end{proof}

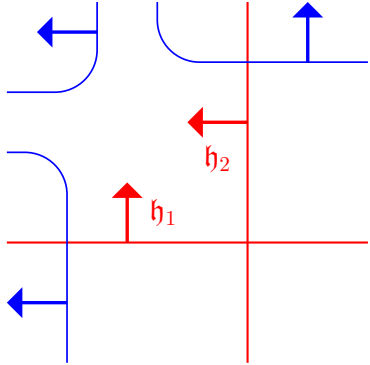
\begin{figure}[H]
	\centering
	\scalebox{.8}{
		\begin{tikzpicture}[auto,node distance=2cm,
			thick,every node/.style={},
			every loop/.style={min distance=2cm},
			hull/.style={draw=none},
			]
			\tikzstyle{label}=[draw=none,font=\Large]

			\draw[red,line width=1pt] (-4,0)--(2,0);
			\draw[draw=red,fill=black,-triangle 90, ultra thick](-2,0)--(-2,1);
			
			\draw[red,line width=1pt] (0,-2)--(0,4);
			\draw[draw=red,fill=black,-triangle 90, ultra thick](0,2)--(-1,2);
			
			\draw[blue,rounded corners=20pt](-3,-2)--(-3,1.5)--(-4,1.5);
			\draw[draw=blue,fill=black,-triangle 90, ultra thick](-3,-1)--(-4,-1);
			
			\draw[blue,rounded corners=20pt](-4,2.5)--(-2.5,2.5)--(-2.5,4);
			\draw[draw=blue,fill=black,-triangle 90, ultra thick](-2.5,3.5)--(-3.5,3.5);
			
			\draw[blue,rounded corners=20pt](-1.5,4)--(-1.5,3)--(2,3);
			\draw[draw=blue,fill=black,-triangle 90, ultra thick](1,3)--(1,4);
			
			\node[red,label] at (-1.4,.5){$\h_1$};
			\node[red,label] at (-.5,1.4){$\h_2$};	
				
		\end{tikzpicture}
	}
	\caption{If $\h_1\cap \h_2$ is a depth-0 quarterspace then Lemma \ref{lem:depth0} implies that there are no halfspaces as shown in blue.}\label{fig:depth0}
\end{figure}

Next, we deduce that $X$ contains a depth-0 quarterspace by the following lemma.

\begin{lem}\label{lem:exist0}
	Any shallow quarterspace contains a depth-0 quarterspace.
\end{lem}
\begin{proof}
Let $\h_1\cap\h_2$ be a shallow quarterspace of depth $r>0$. By Lemma \ref{lem:depth0} we may assume there is a halfspace $\h\subsetneq\h_1$ with $\h\cap\h_2\neq\emptyset$. We cannot have $\h\subseteq\h_2$, as then $\h$ being deep (since $X$ is essential) would imply that $\h_1\cap\h_2$ is deep, thus $\h$ and $\h_2$ are transverse. We now claim that $\h\cap\h_2$ is a quarterspace of depth less than $r$ -- the lemma then follows by induction on depth. Indeed for any $x\in\h\cap\h_2$, we have an inclusion
$$\{\h'\in\cH\mid x\in\h',\,\h'\cap\h^*\cap\h_2^*=\emptyset\}\subsetneq\{\h'\in\cH\mid x\in\h',\,\h'\cap\h_1^*\cap\h_2^*=\emptyset\},$$
which is strict because $\h_1$ is in the second set but not the first ($\h_1\cap\h^*\cap\h_2^*\neq\emptyset$ by Lemma \ref{lem:intersect}). Lemma \ref{lem:d} then implies that
$$d(x,\h^*\cap\h_2^*)<d(x,\h_1^*\cap\h_2^*)\leq r,$$
so we conclude that $\h\cap\h_2$ has depth at most $r-1$.
\end{proof}

\subsection{The pocset $(\cH/\sim,\leq)$}

The cube complex $Y$ will be constructed from a modified version of the pocset of halfspaces $(\cH,\subseteq)$. We first define a \emph{quasi-order} $\leq$ on $\cH$ -- i.e. a binary relation that is reflexive and transitive but may have $\h_1\leq\h_2\leq\h_1$ for $\h_1\neq\h_2$. We define this by $\h_1\leq\h_2$ if $\h_1\subseteq\h_2$ or $\h_1\cap\h_2^*$ is a depth-0 quarterspace. We will make use of the following equivalent formulation.

\begin{lem}\label{lem:leq}
	$\h_1\leq\h_2$ if and only if any halfspace $\h\subsetneq\h_1$ (resp. $\h\subsetneq\h_2^*$) satisfies $\h\subsetneq\h_2$ (resp. $\h\subsetneq\h_1^*$).
\end{lem}
\begin{proof}
	If $\h_1,\h_2$ are transverse then this equivalence reduces to Lemma \ref{lem:depth0}. If $\h_1\subseteq\h_2$ then both conditions are clearly satisfied. If $\h_2\subsetneq\h_1$ or $\h_1\cap\h_2=\emptyset$ then it is easy to see that neither condition is satisfied (noting that $\exists\h\subsetneq\h_1$ by Lemma \ref{lem:d} and essentialness of $X$).
\end{proof}

\begin{lem}
	$\leq$ is a quasi-order on $\cH$.
\end{lem}
\begin{proof}
	Reflexivity is immediate from the definition. Transitivity follows easily from Lemma \ref{lem:leq}: if $\h_1\leq\h_2\leq\h_3$ and $\h\subsetneq\h_1$ then $\h_1\leq\h_2$ implies $\h\subsetneq\h_2$ and $\h_2\leq\h_3$ implies $\h\subsetneq\h_3$; similarly $\h\subsetneq\h_3^*$ implies $\h\subsetneq\h_1^*$.
\end{proof}

To turn $\leq$ into a partial order we quotient $\cH$ by the equivalence relation $\h_1\sim\h_2$ if $\h_1\leq\h_2\leq\h_1$. Let $[\h]$ denote the equivalence class of $\h$. If $\h_1\sim\h_2$, then we cannot have $\h_1\subsetneq\h_2$ as $\h_2\leq\h_1$ would imply $\h_1\subsetneq\h_1$, similarly we cannot have $\h_2\subsetneq\h_1$. It follows that the elements within an equivalence class $[\h]$ are pairwise transverse, and so the size of $[\h]$ is bounded by the dimension of $X$ (note that $X$ is finite dimensional since it admits a cocompact group action). 

\begin{figure}[H]
	\centering
	\scalebox{.8}{
		\begin{tikzpicture}[auto,node distance=2cm,
			thick,every node/.style={},
			every loop/.style={min distance=2cm},
			hull/.style={draw=none},
			]
			\tikzstyle{label}=[draw=none,font=\Large]
			
			\begin{scope}[scale=.8]
			\draw(-5,6)--(0,6);
			\draw(-5,4)--(0,4);
			\draw(-5,2)--(2,2);
			\draw(-4,7)--(-4,2);
			\draw(-2,7)--(-2,2);
			\draw(0,7)--(0,0);
			\draw(2,2)--(2,-3);
			\draw(0,0)--(5,0);
			\draw(2,-2)--(5,-2);
			\draw(4,0)--(4,-3);
			
			\draw[red,line width=1pt] (0,1)--(2,1);
			\draw[draw=red,fill=black,-triangle 90, ultra thick](1,1.5)--(.3,1.5);
			
			\draw[red,line width=1pt] (1,0)--(1,2);
			\draw[draw=red,fill=black,-triangle 90, ultra thick](1.5,1)--(1.5,1.7);
			
			\node[red,label] at (.6,2.4){$\h_1$};
			\node[red,label] at (2.4,1.3){$\h_2$};	
			\end{scope}
			
			\begin{scope}[shift={(10,0)},scale=.8]
\draw(-5,6)--(0,6);
\draw(-5,4)--(0,4);
\draw(-5,2)--(0,2);
\draw(-4,7)--(-4,2);
\draw(-2,7)--(-2,2);
\draw(0,7)--(0,2);
\draw(2,0)--(2,-3);
\draw(2,0)--(5,0);
\draw(2,-2)--(5,-2);
\draw(4,0)--(4,-3);
\draw(0,2)--(2,0);
\node[circle,red,draw,fill,inner sep=0pt,minimum size=5pt] at (1,1){};
\draw[draw=red,fill=black,-triangle 90, ultra thick](1,1)--(.3,1.7);
\node[red,label] at (2.2,1.9){$[\h_1]=[\h_2]$};
			\end{scope}	
		\end{tikzpicture}
	}
	\caption{Example of $\h_1\sim\h_2$ in $X$ on the left. The cube complex $Y=C(\cH/\sim)$ on the right.}\label{fig:h1simh2}
\end{figure}
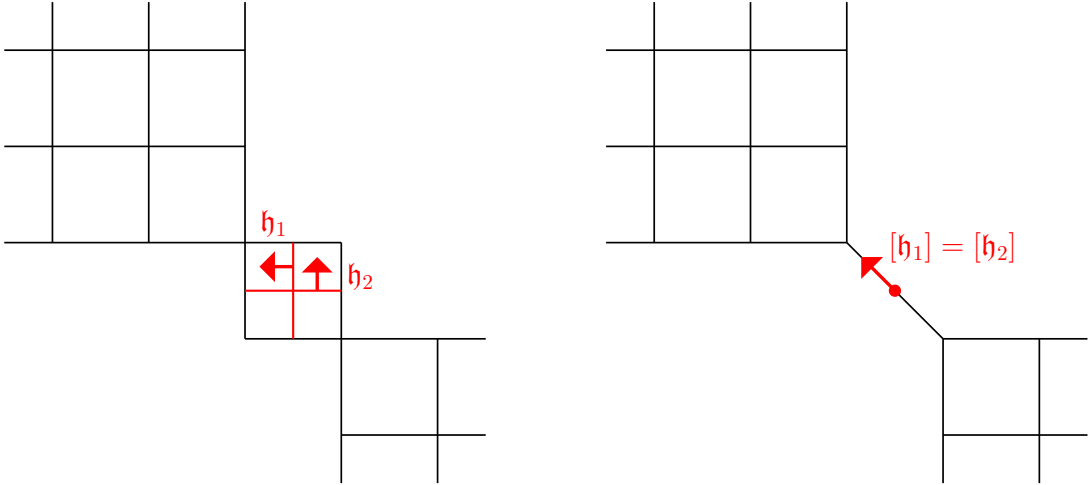

It follows straight from the definition of $\leq$ that $\h_1\leq\h_2$ if and only if $\h_2^*\leq\h_1^*$. 
It is also immediate that $\h_1,\h_2$ are $\subseteq$-transverse whenever $[\h_1],[\h_2]$ are $\leq$-transverse.
Putting this all together we get the following lemma.

\begin{lem}\label{lem:pocsetquotient}
	$(\cH/\sim,\leq)$ is a pocset with involution defined by $[\h]^*:=[\h^*]$. Moreover, the width of $(\cH/\sim,\leq)$ is at most $\dim X$; and the quotient map $\cH\to\cH/\sim$ defines a pocset map $q:(\cH,\subseteq)\to(\cH/\sim,\leq)$, with sizes of fibers bounded by $\dim X$.
\end{lem}

\subsection{The cubing $Y$}

The cube complex $Y$ will be the cubing of $(\cH/\sim,\leq)$. To show that this is well-defined we must find a DCC ultrafilter on $(\cH/\sim,\leq)$.
Our arguments will mainly be on the level of ultrafilters for the remainder of this section, so we will consider vertices $\omega\in X^0$ as DCC ultrafilters on $(\cH,\subseteq)$ (using Theorem \ref{thm:duality}) and vertices $\mu\in Y^0$ as DCC ultrafilters on $(\cH/\sim,\leq)$ (once we know that they exist!).

Given $\omega\in X^0$, consider the partition $\omega=\omega^0\sqcup\omega^1$, where $\h_1\in\omega^0$ if there exists $\h_2\in\omega$ with $\h_1\cap\h_2$ a depth-0 quarterspace.

\begin{lem}\label{lem:omegax0}
	$|\omega^0|\leq\dim X$ for all $\omega\in X^0$.
\end{lem}
\begin{proof}
	We show that the halfspaces in $\omega^0$ are pairwise transverse.
	Firstly, if $\h_1\cap\h_2$ is a depth-0 quarterspace for $\h_1,\h_2\in\omega$, then Lemma \ref{lem:depth0} together with the consistency of $\omega$ implies that $\h_1,\h_2$ are both $\subseteq$-minimal in $\omega$.
	So every halfspace in $\omega^0$ is $\subseteq$-minimal in $\omega$.
	To prove the lemma it suffices to consider distinct $\h_1,\h_2,\h_3\in\omega^0$ such that $\h_1\cap\h_2$ is a depth-0 quarterspace, and show that $\h_1,\h_3$ are transverse.
	Indeed, $\subseteq$-minimality of $\h_1,\h_3$ in $\omega$ implies that we cannot have $\h_1\subsetneq\h_3$ or $\h_3\subsetneq \h_1$. And if $\h_3^*\subsetneq\h_1$ then $\h_3^*\subsetneq\h_2^*$ by Lemma \ref{lem:depth0}, so $\h_2\subsetneq\h_3$, contradicting $\subseteq$-minimality of $\h_3$.
\end{proof}

If $\h_1,\h_2\in\omega^1$ then $\h_1\not\subseteq\h_2^*$ and $\h_1\cap\h_2$ is not a depth-0 quarterspace, hence $\h_1\nleq\h_2^*$. Thus $\omega^1$ pushes forward to a partial ultrafilter $q_*\omega^1$ on $(\cH/\sim,\leq)$, given by
$$q_*\omega^1:=\{[\h]\mid\h\in\omega^1\}.$$

\begin{lem}
	$q_*\omega^1$ is DCC and cofinite.
\end{lem}
\begin{proof}
	Suppose for contradiction that $[\h_1]>[\h_2]>[\h_3]>\cdots$ is a strictly descending infinite chain in $q_*\omega^1$, with $\h_i\in\omega^1$. For each $i<j$, either $\h_i\supsetneq\h_j$ or $\h_j\cap\h_i^*$ is a depth-0 quarterspace. In the latter case $\h_i,\h_j$ are transverse, and a collection of pairwise transverse halfspaces can be no larger than $\dim X$, so it follows from Ramsey's Theorem that there is an infinite $\subsetneq$-descending subsequence of the $\h_i$. This contradicts $\omega$ being a DCC ultrafilter. Hence $q_*\omega^1$ is DCC.
	
	To see cofiniteness of $q_*\omega^1$, consider a pair $[\h],[\h^*]\notin q_*\omega^1$. It follows that $\h,\h^*\notin\omega^1$, so one of them must lie in $\omega^0$, but then there are at most $\dim X$ possibilities for $\h,\h^*$ by Lemma \ref{lem:omegax0}.
\end{proof}

By Lemma \ref{lem:extend}, $q_*\omega^1$ can be extended to a DCC complete ultrafilter $\theta(\omega)$ on $(\cH/\sim,\leq)$. 
This shows the existence of DCC complete ultrafilters on $(\cH/\sim,\leq)$, and it also gives us a map $\theta: X^0\to Y^0$ (albeit not a canonical one!).

We will see later that $\theta$ is actually a quasi-isometry.
We now define a map $\phi:Y^0\to X^0$, which will turn out to be a coarse inverse of $\theta$.
Given $\mu\in Y^0$, define
$$\phi(\mu):=\{\h\in\cH\mid[\h]\in\mu\}.$$

\begin{lem}
	$\phi(\mu)$ is a DCC ultrafilter on $(\cH,\subseteq)$, and the map $\phi:Y^0\to X^0$ is injective.
\end{lem}
\begin{proof}
	For a halfspace $\h\in\cH$ we have exactly one of $[\h],[\h]^*=[\h^*]$ in $\mu$ by completeness of $\mu$, so we have exactly one of $\h,\h^*$ in $\phi(\mu)$.
	If $\h_1\subseteq\h_2$ with $\h_1\in\phi(\mu)$, then $[\h_1]\in\mu$ and $[\h_1]\leq[\h_2]$, so $[\h_2]\in\mu$ by consistency of $\mu$. Thus $\h_2\in\phi(\mu)$.
	This shows that $\phi(\mu)$ is an ultrafilter on $(\cH,\subseteq)$. 
	Any strictly $\subseteq$-descending chain in $\phi(\mu)$ projects to a strictly $\leq$-descending chain in $\mu$, so $\mu$ being DCC implies that $\phi(\mu)$ is DCC.
	Finally, the map $\phi:Y^0\to X^0$ is injective, because for distinct $\mu_1,\mu_2\in Y^0$ there exists $[\h]\in\mu_1-\mu_2$, so $\h\in\phi(\mu_1)-\phi(\mu_2)$.
\end{proof}

Apart from the map $\theta$, all the constructions so far are entirely canonical, so the action of $G$ on $X$ induces actions of $G$ on $(\cH/\sim,\leq)$ and $Y$, and the map $\phi$ is $G$-equivariant.

\begin{lem}
	$G$ acts cocompactly on $Y$.
\end{lem}
\begin{proof}
By Lemma \ref{lem:cocompact} it is equivalent to show that there are finitely many $G$-orbits of collections of pairwise transverse elements in $(\cH/\sim,\leq)$.
Now $q:(\cH,\subseteq)\to(\cH/\sim,\leq)$ is a surjective pocset map with pairwise transverse fibers, so the preimage of a collection of pairwise transverse elements is pairwise transverse.
But $G$ acts cocompactly on $X$, so there are finitely many collections of pairwise transverse elements in $(\cH,\subseteq)$. The result follows.
\end{proof}

\begin{lem}\label{lem:notsurj}
	$\phi$ is not surjective. In particular $Y$ has fewer $G$-orbits of vertices than $X$.
\end{lem}
\begin{proof}
	If $\h_1\cap\h_2$ is a depth-0 quarterspace then $\h_1\leq\h_2^*$, so no $\mu\in Y^0$ has $[\h_1],[\h_2]\in\mu$, which means no $\mu\in Y^0$ has $\h_1,\h_2\in\phi(\mu)$. But any vertex in the quarterspace $\h_1\cap\h_2$ is represented by an ultrafilter $\omega\in X^0$ with $\h_1,\h_2\in\omega$. We know that $X$ does contain depth-0 quarterspaces by Lemma \ref{lem:exist0}, so we conclude that $\phi$ is not surjective.
\end{proof}

\begin{remk}
	It is not hard to extend the arguments in Lemma \ref{lem:notsurj} to show that $X^0-\phi(Y^0)$ is precisely the union of depth-0 quarterspaces in $X$.
\end{remk}

Now we turn to showing that $\phi$ is a quasi-isometry.

\begin{lem}\label{lem:QIembed}
	$d(\mu_1,\mu_2)\leq d(\phi(\mu_1),\phi(\mu_2))\leq (\dim X) d(\mu_1,\mu_2)$ for all $\mu_1,\mu_2\in Y^0$.
\end{lem}
\begin{proof}
	By Lemma \ref{lem:domega} this is equivalent to showing
	$$|\mu_1\triangle\mu_2|\leq |\phi(\mu_1)\triangle\phi(\mu_2)|\leq(\dim X) |\mu_1\triangle\mu_2|.$$
	And this follows from
	\begin{align*}
		\phi(\mu_1)\triangle\phi(\mu_2)=\{\h\in\cH\mid[\h]\in\mu_1\triangle\mu_2\}
	\end{align*}
and the fact that the classes $[\h]$ have size at most $\dim X$.
\end{proof}

We now show that $\theta$ is a coarse inverse to $\phi$, so we conclude from Lemma \ref{lem:QIembed} that $\theta$ and $\phi$ are both quasi-isometries.

\begin{lem}\label{lem:thetaphi}
	$\theta\phi$ is the identity map on $Y^0$ and $d(\phi\theta(\omega),\omega)\leq2\dim X$ for all $\omega\in X^0$.
\end{lem}
\begin{proof}
	Let $\mu\in Y^0$. We claim that $\phi(\mu)^0=\emptyset$. Indeed $\phi(\mu)^0\neq\emptyset$ would imply the existence of $\h_1,\h_2\in \phi(\mu)$ such that $\h_1\cap\h_2$ is a depth-0 quarterspace. But then $\h_1\leq\h_2^*$, so $[\h_1]\leq[\h_2]^*$, contradicting the consistency of $\mu$. Hence $\phi(\mu)=\phi(\mu)^1$, so
	$$q_*\phi(\mu)^1=q_*\phi(\mu)=\mu$$
	is already a complete ultrafilter, thus $\theta\phi(\mu)=\mu$.
	
	For the second part of the lemma, we note that $q_*\omega^1\subseteq\theta(\omega)$, so $\omega^1\subseteq\phi\theta(\omega)$. Then applying Lemma \ref{lem:domega} yields 
	$$d(\phi\theta(\omega),\omega)=2|\omega-\phi\theta(\omega)|\leq2|\omega^0|\leq2\dim X.\qedhere$$
\end{proof}

Next we show that each halfspace in $Y$ is mapped by $\phi$ to within finite Hausdorff distance of a halfspace in $X$. 
By Proposition \ref{prop:cubing}\ref{item:halfspace}, halfspaces in $Y$ correspond to the elements of $\cH/\sim$, and it is immediate from the construction of $\phi$ that each halfspace $[\h]$ is mapped within the halfspace $\h$ in $X$. Thus it remains to show that $\h$ is contained in a bounded neighborhood of the image $\phi[\h]$. This follows from Lemma \ref{lem:thetaphi} together with the following lemma.

\begin{lem}
	Let $\h\in\cH$. We can choose the map $\theta$ so that $\phi\theta:X^0\to X^0$ preserves the halfspace $\h$.
\end{lem}
\begin{proof}
	Suppose $\h\in\omega\in X^0$. Our goal is to choose $\theta(\omega)$ so that $[\h]\in\theta(\omega)$, as this would imply that $\h\in\phi\theta(\omega)$. If $\h\in\omega^1$ then $[\h]\in\theta(\omega)$ is automatic because $[\h]\in q_*\omega^1$, so suppose $\h\in\omega^0$. Observe that $q_*\omega^1\cup\{[\h]\}$ is a partial ultrafilter on $(\cH/\sim,\leq)$: indeed if there is $\h'\in\omega^1$ with $\h'\leq\h^*$, then either $\h'\subseteq\h^*$, contradicting consistency of $\omega$, or $\h'\cap\h$ is a depth-0 quarterspace, contradicting $\h'\in\omega^1$. Moreover, the fact that $q_*\omega^1$ is DCC and cofinite implies that $q_*\omega^1\cup\{[\h]\}$ is DCC and cofinite. Thus we may choose  $\theta(\omega)$ to be an extension of $q_*\omega^1\cup\{[\h]\}$ (still using Lemma \ref{lem:extend}).
\end{proof}

Finally, we prove that our construction preserves local finiteness.

\begin{lem}\label{lem:locfinite}
	$Y$ is locally finite if $X$ is locally finite.
\end{lem}
\begin{proof}
	Let $\mu\in Y^0$ and let $[\h_1]\in\mu$ be $\leq$-minimal in $\mu$. If $\h_2\in\phi(\mu)$ satisfies $\h_2\subseteq\h_1$ then $[\h_2]\in\mu$ and $[\h_2]\leq[\h_1]$, so by minimality of $[\h_1]$ we have $[\h_2]=[\h_1]$. But we saw earlier that elements of a $\sim$-equivalence class are pairwise transverse, so in fact $\h_2=\h_1$. Hence $\h_1$ is $\subseteq$-minimal in $\phi(\mu)$. It follows from Lemma \ref{lem:edges} that the number of edges in $Y$ incident to $\mu$ is at most the number of edges in $X$ incident to $\phi(\mu)$, and there are finitely many such edges since $X$ is locally finite.
\end{proof}

\bigskip
\section{Semistability at infinity}\label{sec:semistability}

In this section we prove Theorem \ref{thm:semistable}, that all cubulated groups are semistable at infinity.
Dunwoody's accessibility (Theorem \ref{thm:accessibility}) together with the following two theorems allow us to reduce to the case of one-ended cubulated groups (noting that finite groups are semistable).

\begin{thm}\cite[Theorem 1]{MihalikTschantz92b}\label{thm:combinesemistable}\\
	If $G=A*_H B$ is an amalgamated product where $A$ and $B$ are finitely presented and semistable at infinity, and $H$ is finitely generated, then $G$ is semistable at infinity.
	If $G=A*_H$ is an HNN-extension where $A$ is semistable at infinity, and $H$ is finitely generated, then $G$ is semistable at infinity.
\end{thm}

\begin{thm}\label{thm:cubulatevgroup}
	Let $G$ be a cubulated group that admits a finite splitting over finite subgroups. Then each vertex group is cubulated.
\end{thm}
\begin{proof}
	$G$ is hyperbolic relative to its infinite vertex groups, so each infinite vertex group is cubulated by \cite[Theorem 1.1]{SageevWise15}. The finite vertex groups are also cubulated because they admit proper cocompact actions on a point.
\end{proof}

We will also make use of the following theorem, which follows straight from the proof of \cite[Theorem 2.1]{Mihalik83} (replacing $\tilde{X}$ by $X$). This provides a characterization for semistability in the one-ended case in terms of ``pushing out'' loops along a fixed proper ray.

\begin{thm}\label{thm:semistableequiv}
	Let $X$ be a one-ended locally finite CW-complex and let $r:[0,\infty)\to X$ be a proper ray. Then the following are equivalent:
	\begin{enumerate}
		\item $X$ is semistable at infinity.
		\item\label{item:push} For any compact set $C$, there is a compact set $D$ such that for any third compact set $E$ and loop $\alpha$ based on $r$ with image in $X-D$, $\alpha$ is homotopic rel$\{r\}$ to a loop in $X-E$, by a homotopy in $X-C$.
	\end{enumerate}
\end{thm}

We are now ready to prove Theorem \ref{thm:semistable}.

\begin{proof}[Proof of Theorem \ref{thm:semistable}]
	Cubulated groups are finitely presented, so by Dunwoody's accessibility (Theorem \ref{thm:accessibility}) we know that any cubulated group admits a finite splitting over finite groups with vertex groups that are either finite or one-ended.
	If these vertex groups are semistable at infinity then Theorem \ref{thm:combinesemistable} implies that the whole group is semistable at infinity.
	Theorem \ref{thm:cubulatevgroup} tells us that the vertex groups in such a splitting are cubulated, so we reduce to proving semistability for finite or one-ended cubulated groups. Finite groups are automatically semistable at infinity, so it suffices to consider the one-ended case.
	
	Let $G$ be a one-ended cubulated group, acting properly and cocompactly on a CAT(0) cube complex $X$. By applying Theorem \ref{thm:halfquarter}, we may assume that all halfspaces in $X$ are one-ended and that all quarterspaces in $X$ are deep.
	We now prove that $X$ (and hence $G$) is semistable at infinity using the characterization in Theorem \ref{thm:semistableequiv}.
	For this proof we will work with the CAT(0) metric on $X$ rather than the combinatorial metric, so we will not consider halfspaces as collections of vertices but instead we will consider them as the CAT(0) convex subspaces of $X$ that arise as complementary components of hyperplanes; and we will refer to the union of a halfspace with its bounding hyperplane as a \emph{closed halfspace}.	
	Let $r$ be a geodesic ray in $X$ and let $C\subseteq X$ be compact. Let $D$ be the intersection of all closed halfspaces containing $C$, which is compact by Lemma \ref{lem:finhull}.
	Let $E\subseteq X$ be a third compact set, and let $\alpha$ be a loop in $X-D$ based on $r$.
	Passing to the cubical neighborhood, we may assume that $E$ is a subcomplex of $X$.
	Since $D$ is an intersection of closed halfspaces, $\alpha$ can be written as a concatenation of paths $\alpha_0,\alpha_1,...,\alpha_n$, such that $\alpha_i$ is contained in a halfspace $\h_i$ disjoint from $D$ (see Figure \ref{fig:push}).
	Let $x$ be the point at the beginning of $\alpha_0$ and the end of $\alpha_n$, and assume that $\alpha$ is based on $r$ at $x$. Also assume that $\h_0=\h_n$. For $0\leq i<n$, let $x_i$ be the point at the end of the segment $\alpha_i$ and the beginning of the segment $\alpha_{i+1}$.
	Observe that $x_i\in\h_i\cap\h_{i+1}$ and $D\subseteq\h_i^*\cap\h_{i+1}^*$, so $\h_i,\h_{i+1}$ are transverse and $\h_i\cap\h_{i+1}$ is a quarterspace. (Note that a pair of halfspaces intersect in the CAT(0) setting if and only if they intersect in the combinatorial setting).
	Each $\h_i$ is one-ended, so only one of the finitely many components of $\h_i-E$ is unbounded, call this component $E_i^*$. Each quarterspace $\h_i\cap\h_{i+1}$ is deep, so in particular unbounded, and $E_i^*\cap E_{i+1}^*\subseteq\h_i\cap\h_{i+1}$ is unbounded.
	
	We now construct a loop $\beta$ in $X-E$, and a homotopy rel$\{r\}$ in $X-D$ from $\alpha$ to $\beta$ (in particular this homotopy is in $X-C$).
	Let $y$ be a point on $r$ in $E_0^*$, and let $y_i$ be a point in $E_i^*\cap E_{i+1}^*$ for $0\leq i<n$.
	Let $\beta_0,\beta_1,...,\beta_n$ be paths in $E_0^*,E_1^*,...,E_n^*$ respectively that join the points $y,y_0,y_1,...,y_{n-1},y$. Let $\beta$ be the concatenation of the $\beta_i$, which lies in $X-E$.
	For $0\leq i\leq n$, we know that $\alpha_i$ and $\beta_i$ both lie in the halfspace $\h_i$, so we can homotope $\alpha_i$ to $\beta_i$ via geodesics, and this homotopy will be in $\h_i$ since halfspaces are convex. By uniqueness of geodesics in CAT(0) spaces, these homotopies will fit together to give a homotopy from $\alpha$ to $\beta$ in $X-D$. Moreover, the homotopy will move the point $x$ along a subsegment of $r$ to $y$, so the homotopy is rel$\{r\}$ as required.
\end{proof}

\begin{figure}[H]
	\centering
	\scalebox{.8}{
		\begin{tikzpicture}[auto,node distance=2cm,
			thick,every node/.style={},
			every loop/.style={min distance=2cm},
			hull/.style={draw=none},
			]
			\tikzstyle{label}=[draw=none,font=\Large]
			
			\newcommand{\toppart}[1]
			{
				\draw[red, line width=1pt] (-6,5)edge [bend right=30](6,5);
				\draw[draw=red,fill=black,-triangle 90, ultra thick](0,3.25)--(0,4);
				
				\draw[blue] (2.6,4.5) edge [blue,line width=1pt,
				postaction={decoration={markings,mark=at position 0.55 with {\arrow[blue,line width=.5mm]{triangle 60}}},decorate}](-2.6,4.5);
				
				\draw[Green] (3.9,6.75) edge [Green,line width=1pt,
				postaction={decoration={markings,mark=at position 0.55 with {\arrow[Green,line width=.5mm]{triangle 60}}},decorate}](-3.9,6.75);
				
			}
			
			\node[circle,fill,inner sep=0pt,minimum size=5pt] at (.6,0) {};
			\draw[draw=black,fill=black,-triangle 90, ultra thick] (.6,0) -- (.6,8);	
			\draw[rounded corners=20] (-2,-1.5) rectangle (2,1.5);

			\begin{scope}[rotate=60]
				\toppart{2};
			\end{scope}
			\begin{scope}[rotate=-60]
				\toppart{2};
			\end{scope}	
			\begin{scope}[rotate=120]
				\toppart{2};
			\end{scope}
			\begin{scope}[rotate=180]
				\draw[blue] (2.6,4.5)-- (1,4.5);
				\draw[blue,dotted](1,4.5)--(0,4.5);
				
				\draw[Green] (3.9,6.75) --(1,6.75);
				\draw[Green,dotted](1,6.75)--(0,6.75);
			\end{scope}
			\begin{scope}[rotate=-120]
				\draw[blue] (-2.6,4.5)-- (-1,4.5);
				\draw[blue,dotted](-1,4.5)--(0,4.5);
				
				\draw[Green] (-3.9,6.75) --(-1,6.75);
				\draw[Green,dotted](-1,6.75)--(0,6.75);
			\end{scope}
			
			\draw[red, line width=1pt] (-6,5)edge [bend right=30](6,5);
			\draw[draw=red,fill=black,-triangle 90, ultra thick](0,3.25)--(0,4);
			
			\node[circle,fill=blue,inner sep=0pt,minimum size=5pt] at (.6,4.5) {};
			\draw[blue] (2.6,4.5) edge [blue,line width=1pt,
			postaction={decoration={markings,mark=at position 0.55 with {\arrow[blue,line width=.5mm]{triangle 60}}},decorate}](.6,4.5);
			\draw[blue] (.6,4.5) edge [blue,line width=1pt,
			postaction={decoration={markings,mark=at position 0.55 with {\arrow[blue,line width=.5mm]{triangle 60}}},decorate}](-2.6,4.5);
			
			\node[circle,fill=Green,inner sep=0pt,minimum size=5pt] at (.6,6.75) {};
			\draw[Green] (3.9,6.75) edge [Green,line width=1pt,
			postaction={decoration={markings,mark=at position 0.55 with {\arrow[Green,line width=.5mm]{triangle 60}}},decorate}](.6,6.75);
			\draw[Green] (.6,6.75) edge [Green,line width=1pt,
			postaction={decoration={markings,mark=at position 0.55 with {\arrow[Green,line width=.5mm]{triangle 60}}},decorate}](-3.9,6.75);
			
			\node[label,red] at (-.6,3.7){$\h_0$};
			\node[label,red] at (-3.5,1.4){$\h_1$};
			\node[label,red] at (-2.8,-2.4){$\h_2$};
			\node[label,red] at (3.8,1.3){$\h_{n-1}$};
			
			\node[label,blue] at (-.8,4.9){$\alpha_0$};
			\node[label,blue] at (-4.5,2.3){$\alpha_1$};
			\node[label,blue] at (-4.4,-2.5){$\alpha_2$};
			\node[label,blue] at (4.7,2.3){$\alpha_{n-1}$};
			\node[label,blue] at (2,4.9){$\alpha_n$};
			\node[label,blue] at (.9,4.8){$x$};
			
			\node[label,Green] at (-2,7.2){$\beta_0$};
			\node[label,Green] at (-6.5,3.3){$\beta_1$};
			\node[label,Green] at (-6.3,-3.7){$\beta_2$};
			\node[label,Green] at (6.6,3.5){$\beta_{n-1}$};
			\node[label,Green] at (2.5,7.2){$\beta_n$};
			\node[label,Green] at (.9,7.1){$y$};
			
			\node[label] at (-.5,1){$C$};
			\node[label] at (0,7.8){$r$};
			
			\node[regular polygon,
			draw=black,
			regular polygon sides = 12,
			text = blue,
			minimum size = 6cm] (p) at (0,0) {};
			\node[label] at (-.5,2.4){$D$};
			
			\node[regular polygon,
			draw=black,
			regular polygon sides = 12,
			text = blue,
			minimum size = 12.7cm] (p) at (0,0) {};
			\node[label] at (-.5,5.7){$E$};
			
		\end{tikzpicture}
	}
	\caption{The proof of Theorem \ref{thm:semistable}.}\label{fig:push}
\end{figure}
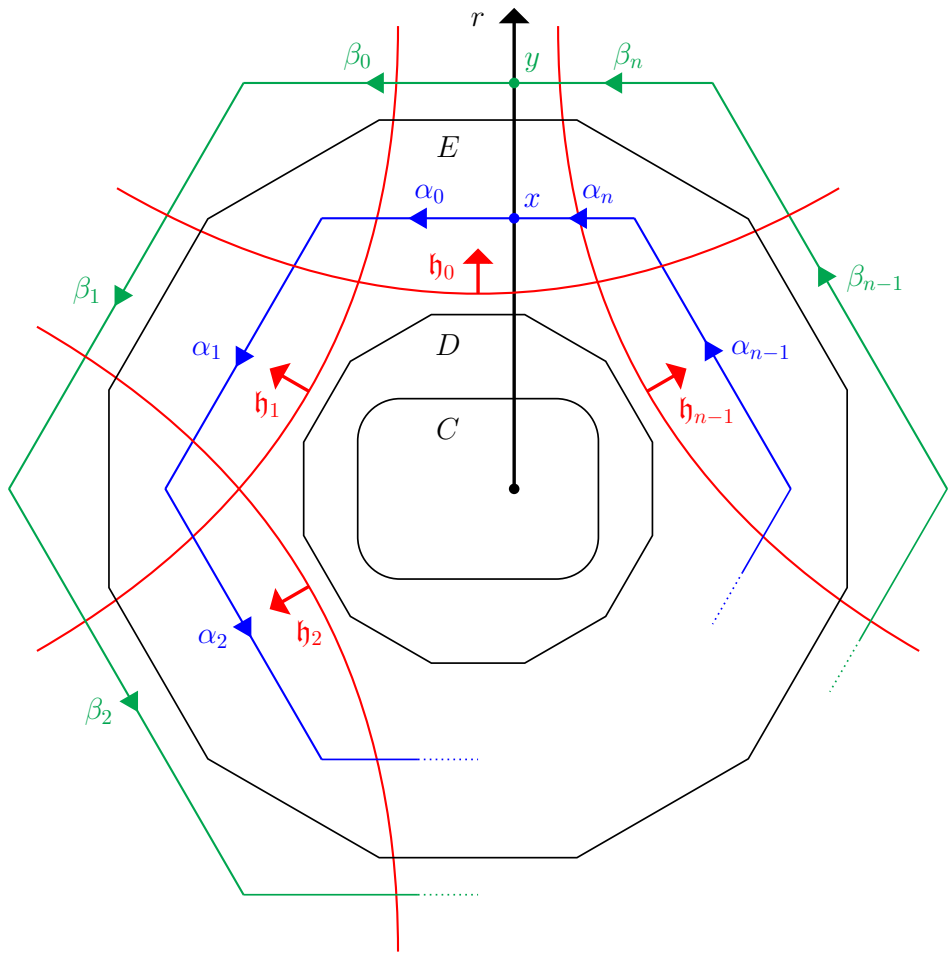

\begin{remk}
	The above proof only requires quarterspaces to be unbounded rather than deep, so Theorem \ref{thm:halfquarter} is actually stronger than what we need to prove Theorem \ref{thm:semistable}.
\end{remk}

\bigskip
\section{Example with an infinite-ended halfspace}\label{sec:example}

In this section we give an example of a one-ended group with a cubulation given by a CAT(0) cube complex that is essential and contains an infinite-ended halfspace.
This demonstrates that Theorem \ref{thm:halfspaces} is not vacuous, and that it requires more than simply passing to the essential core of a CAT(0) cube complex.

Consider the following cyclic amalgam of free groups $\mathbb{F}_m$ and $\mathbb{F}_n$

\begin{equation}\label{amalgam}
G=\mathbb{F}_m\ast_{\mathbb{Z}}\mathbb{F}_n  =  \langle \mathbb{F}_m,\mathbb{F}_n \mid w_1 = w_2 \rangle,
\end{equation}

where $w_1\in\mathbb{F}_m$ and $w_2\in \mathbb{F}_n$.
If $w_1,w_2$ are cyclically reduced and have the same length $L$ with respect to the standard generators of $\mathbb{F}_m$ and $\mathbb{F}_n$, then we may construct a non-positively curved square complex $X$ with fundamental group $G$ as follows.
Take graphs $R_1,R_2$ with one vertex each and $m,n$ edges respectively (the \emph{roses with $m$ or $n$ petals}); $R_1,R_2$ have fundamental groups $\mathbb{F}_m,\mathbb{F}_n$ respectively, where each edge corresponds to a generator.
Take an annulus $A$ formed by identifying the top and bottom of a $2\times L$ square grid.
Then form the square complex $X$ by attaching the left-hand boundary of $A$ to $R_1$ along the word $w_1$ and attaching the right-hand boundary of $A$ to $R_2$ along $w_2$ (see Figure \ref{fig:amalgam}).
$X$ is non-positively curved because $w_1,w_2$ are cyclically reduced.
$X$ has the structure of a graph of spaces with vertex spaces $R_1,R_2$ and edge space $A$; and this structure corresponds to the splitting (\ref{amalgam}) for $G$, so $G=\pi_1(X)$.

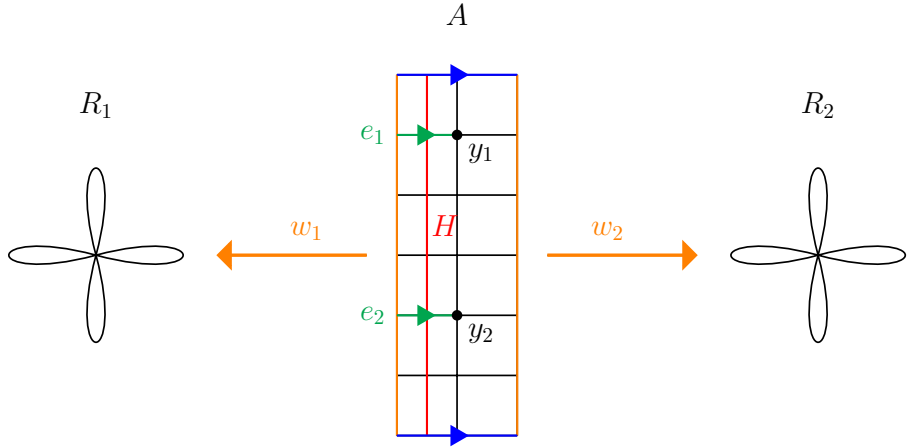
\begin{figure}[H]
	\centering
	\scalebox{.8}{
		\begin{tikzpicture}[auto,node distance=2cm,
			thick,every node/.style={},
			every loop/.style={min distance=2cm},
			hull/.style={draw=none},
			]
			\tikzstyle{label}=[draw=none,font=\Large]
			
			\begin{scope}[shift={(-5,3)}]
				\path
				(0,0) edge [loop left] (0,0)
				(0,0) edge [loop right] (0,0)
				(0,0) edge [loop below] (0,0)
				(0,0) edge [loop above] (0,0);
				
				\node[label] at (0,2.5) {$R_1$};
			\end{scope}
			
			\begin{scope}	
				\draw(0,0) grid (2,6);
				\draw[draw=orange,fill=black,-triangle 90, ultra thick] (-0.5,3) -- (-3,3);
				\draw[draw=orange,fill=black,-triangle 90, ultra thick] (2.5,3) -- (5,3);
				
				\path
				(0,0) edge [orange,line width=1pt] (0,6)
				(2,0) edge [orange,line width=1pt] (2,6)
				(0.5,0) edge [red,line width=1pt] (0.5,6)
				(0,5) edge [Green,line width=1pt,
				postaction={decoration={markings,mark=at position 0.65 with {\arrow[Green,line width=.5mm]{triangle 60}}},decorate}] (1,5)
				(0,2) edge [Green,line width=1pt,
				postaction={decoration={markings,mark=at position 0.65 with {\arrow[Green,line width=.5mm]{triangle 60}}},decorate}] (1,2)
			    (0,6) edge [blue,line width=1pt,
				postaction={decoration={markings,mark=at position 0.6 with {\arrow[blue,line width=.5mm]{triangle 60}}},decorate}] (2,6)
				(0,0) edge [blue,line width=1pt,
				postaction={decoration={markings,mark=at position 0.6 with {\arrow[blue,line width=.5mm]{triangle 60}}},decorate}] (2,0);
				
				\node[label] at (1,7) {$A$};
				\node[label,orange] at (-1.5,3.4) {$w_1$};
				\node[label,orange] at (3.5,3.4) {$w_2$};
				\node[label,Green] at (-.4,5) {$e_1$};
				\node[label,Green] at (-.4,2) {$e_2$};
				\node[label,red] at (.8,3.5) {$H$};
				\node[circle,fill,inner sep=0pt,minimum size=5pt] at (1,5) {};
				\node[circle,fill,inner sep=0pt,minimum size=5pt] at (1,2) {};
				\node[label] at (1.4,4.7) {$y_1$};
				\node[label] at (1.4,1.7) {$y_2$};
			\end{scope}
		
		\begin{scope}[shift={(7,3)}]
			\path
			(0,0) edge [loop left] (0,0)
			(0,0) edge [loop right] (0,0)
			(0,0) edge [loop below] (0,0)
			(0,0) edge [loop above] (0,0);
			
			\node[label] at (0,2.5) {$R_2$};
		\end{scope}

		\end{tikzpicture}
	}
	\caption{Construction of the square complexes $X$ and $X'$.}\label{fig:amalgam}
\end{figure}

The group $G$ might be one-ended or infinite-ended. 
For example, if one of the generators $a$ of $\mathbb{F}_m$ does not appear in the word $w_1$ then $G$ admits a free splitting with $\langle a\rangle$ as one of the factors, so $G$ is infinite-ended.
As an example where $G$ is one-ended, we can let $\mathbb{F}_m,\mathbb{F}_n$ be rank-2 free groups with generating sets $\{a_1,b_1\},\{a_2,b_2\}$ respectively and take the elements $w_1,w_2$ to be the commutators $[a_1,b_1],[a_2,b_2]$; in this case $X$ is homeomorphic to the oriented surface of genus 2.
As a further example, if $w_1,w_2$ are sufficiently generic elements then $G$ is one-ended and (\ref{amalgam}) is a JSJ splitting for $G$ over cyclic subgroups \cite[Example 2.27]{ShepherdWoodhouse22}.
Henceforth we will assume that $G$ is one-ended.

The universal cover $\tilde{X}$ of $X$ is a CAT(0) cube complex, and the action of $G$ on $\tilde{X}$ by deck transformations is a cubulation of $G$. In particular $\tilde{X}$ is one-ended.
However, $\tilde{X}$ might not contain an infinite-ended halfspace. To exhibit a cubulation of $G$ with an infinite-ended halfspace we will modify $X$ to obtain another non-positively curved square complex $X'$, and then consider the universal cover of $X'$.

The construction of $X'$ requires an additional (but mild) assumption on $w_1$: if $w_1=a_1 a_2\cdots a_k$ as a word in the generators of $\mathbb{F}_m$, then we assume that there exist $1\leq p<q<k$ such that $a_p\neq a_q$ and $a_{p+1}\neq a_{q+1}$.
The construction of $X'$ is then as follows.
Label the edges of the left-hand boundary of the annulus $A$ according to the word $w_1$, and let $e_1,e_2$ be the horizontal edges in $A$ that meet this boundary in the middle of the subwords $a_p a_{p+1}, a_q a_{q+1}$ respectively.
Orient $e_1, e_2$ from left to right (as shown in Figure \ref{fig:amalgam}).
We obtain $X'$ from $X$ by gluing together the (oriented) edges $e_1,e_2$.
Note that the left-hand endpoints of $e_1,e_2$ define the same vertex in $X$, so we can think of the gluing as folding $e_1$ and $e_2$ together from left to right.
The condition that $a_p\neq a_q$ and $a_{p+1}\neq a_{q+1}$ is necessary for $X'$ to be locally CAT(0) at the common left-hand endpoint of $e_1$ and $e_2$.
It is also instructive to consider the effect of the gluing on the hyperplane $H$ dual to $e_1$ and $e_2$: before gluing $H$ is homeomorphic to a circle, but after gluing it becomes a wedge of two circles, since the midpoint of $e_1$ has been identified with the midpoint of $e_2$.
It is not hard to show that the quotient map $X\to X'$ is a homotopy equivalence, so $\pi_1(X')=\pi_1(X)=G$, and the action of $G$ on the universal cover $\tilde{X}'$ of $X'$ is another cubulation of $G$.

\begin{prop}\label{prop:tX'}
	$\tilde{X}'$ is essential and contains an infinite-ended halfspace.
\end{prop}
\begin{proof}
	The horizontal and vertical edge paths in the annulus $A$ map to closed local geodesics in $X'$, and every edge of $X'$ is contained in one of these local geodesics.
	Lifting these local geodesics to $\tilde{X}'$, we see that every edge $\tilde{e}'$ in $\tilde{X}'$ is contained in a bi-infinite geodesic $\tilde{\gamma}'$ (in a CAT(0) space local geodesics are geodesics).
	Furthermore, $\tilde{\gamma}'$ meets the hyperplane $\hat{\h}'$ dual to $\tilde{e}'$ at right-angles, so it follows from basic CAT(0) geometry that $\tilde{\gamma}'$ goes arbitrarily far from $\hat{\h}'$ in both the halfspaces bounded by $\hat{\h}'$. It follows that every halfspace in $\tilde{X}'$ is deep, so $\tilde{X}'$ is essential.
	
	It remains to show that $\tilde{X}'$ contains an infinite-ended halfspace.
	Let $y_1,y_2$ be the right-hand endpoints of the edges $e_1,e_2$ (shown in Figure \ref{fig:amalgam}). 
	In $X'$ the edges $e_1,e_2$ are glued together to form a single edge $e'$ and the vertices $y_1,y_2$ are identified to give a single vertex $y'$.
	Let $\tilde{y}'$ be a lift of $y'$ to $\tilde{X}'$ and let $\tilde{e}'$ be the lift of $e'$ incident at $\tilde{y}'$.
	Let $\hat{\h}'$ be the hyperplane dual to $\tilde{e}$ and let $\h'$ be the halfspace containing $\tilde{y}'$ that is bounded by $\hat{\h}'$.
	(For this proof we consider halfspaces as complementary components of hyperplanes rather than taking the combinatorial viewpoint from Section \ref{subsec:CCC}.)
	The local picture of of $\tilde{X}'$ at $\tilde{y}$ is shown in Figure \ref{fig:local}.	
	Observe that $\tilde{e}'$ cuts $\h'$ into two components in this local picture.
	We claim that $\tilde{e}'$ also cuts $\h'$ into two components globally.
	Indeed, if $\h'-\tilde{e}'$ was connected then $\h'$ would be obtained from $\h'-\tilde{e}'$ by amalgamating the two sides of $\tilde{e}'\cap\h'$, so $\h'$ would have non-trivial fundamental group (given by a HNN extension of $\pi_1(\h'-\tilde{e}')$).
	But this cannot happen since $h'$ is a convex subspace of the CAT(0) space $\tilde{X}'$.
	The other lifts of $\tilde{e}'$ dual to $\hat{\h}'$ also cut $\h'$ into two components, so we conclude that $\h'$ is infinite-ended.
\end{proof}

\begin{figure}[H]
	\centering
	\scalebox{.8}{
		\begin{tikzpicture}[auto,node distance=2cm,
			thick,every node/.style={},
			every loop/.style={min distance=2cm},
			hull/.style={draw=none},
			]
			\tikzstyle{label}=[draw=none,font=\Large]
			
			\draw[fill=red,opacity=.2] (0,.5)--(1,1.5)--(1,1)--(2.5,.8)--(1.5,-.2)--(2.5,-1.2)--(1,-1)--(0,0)--(-1,-1)--(-2.5,-1.2)--(-1.5,-.2)--(-2.5,.8)--(-1,1)--(-1,1.5)--(0,.5);
			
			\draw[dashed] (1,-1) -- (1,0) -- (0,1) -- (-1,0) -- (-1,-1);
			\draw[red,dashed,line width=1pt] (1,-.5)--(0,.5)--(-1,-.5);
			
			\draw (0,0) -- (1,1) -- (1,2) -- (0,1) -- (-1,2) -- (-1,1) -- (0,0);
			\draw (0,0) -- (1.5,-.2) -- (2.5,-1.2) -- (1,-1) -- (0,0);
			\draw (0,0) -- (-1.5,-.2) -- (-2.5,-1.2) -- (-1,-1) -- (0,0);
			\draw (1,1) -- (2.5,.8) -- (1.5,-.2);	
			\draw (-1,1) -- (-2.5,.8) -- (-1.5,-.2);	
				
			\path
			(0,.5) edge [red,line width=1pt] (1,1.5)
			(0,.5) edge [red,line width=1pt] (-1,1.5)
			(0,1) edge [Green,line width=1pt,
				postaction={decoration={markings,mark=at position 0.65 with {\arrow[Green,line width=.5mm]{triangle 60}}},decorate}] (0,0);

				\node[label] at (0,-.5) {$\tilde{y}'$};
				\node[label,Green] at (0,1.4) {$\tilde{e}$};
				\node[label,red] at (1.4,1.6) {$\hat{\h}'$};
				\node[circle,fill,inner sep=0pt,minimum size=5pt] at (0,0) {};			
		\end{tikzpicture}
	}
	\caption{The local picture of $\tilde{X}'$ at $\tilde{y}$. The halfspace $\h'$ is shaded red.
	The left-hand-portion of the picture is a lift of the neighborhood of $y_1$ in the annulus $A$, while the right-hand portion is a lift of the neighborhood of $y_2$ in $A$.}\label{fig:local}
\end{figure}
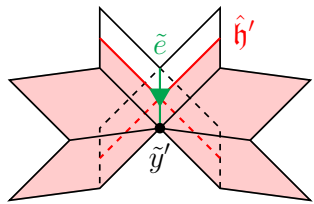

\begin{remk}
	The folding together of $e_1$ and $e_2$ to produce $X'$ is an example of a cubical Stallings fold, as studied in \cite{BeekerLazarovich18,DaniLevcovitz21,BenZviKrophollerLyman22}.
	One can also think of $\tilde{X}'$ as being obtained from $\tilde{X}$ by infinitely many Stallings folds corresponding to the lifts of $e_1$ and $e_2$.
	In fact, this folding is the reverse of the procedure in Section \ref{sec:oneend} that proves Theorem \ref{thm:smallerstabs}.
	More precisely, if we apply the procedure in Section \ref{sec:oneend} to $\tilde{X}'$, with $\h_0$ being the halfspace $\h'$ from the proof of Proposition \ref{prop:tX'} and $\fc_0$ consisting of just the vertex $\tilde{y}'$, then we recover the cube complex $\tilde{X}$. 
\end{remk}

\bibliographystyle{alpha}
\bibliography{Ref}

\begin{thebibliography}{BZKL22}

\bibitem[BH99]{BridsonHaefliger99}
Martin~R. Bridson and Andr\'{e} Haefliger.
\newblock {\em Metric spaces of non-positive curvature}, volume 319 of {\em
  Grundlehren der Mathematischen Wissenschaften [Fundamental Principles of
  Mathematical Sciences]}.
\newblock Springer-Verlag, Berlin, 1999.

\bibitem[BL18]{BeekerLazarovich18}
Benjamin Beeker and Nir Lazarovich.
\newblock Stallings' folds for cube complexes.
\newblock {\em Israel J. Math.}, 227(1):331--363, 2018.

\bibitem[BM91]{BestvinaMess91}
Mladen Bestvina and Geoffrey Mess.
\newblock The boundary of negatively curved groups.
\newblock {\em J. Amer. Math. Soc.}, 4(3):469--481, 1991.

\bibitem[Bow99]{Bowditch99}
B.~H. Bowditch.
\newblock Connectedness properties of limit sets.
\newblock {\em Trans. Amer. Math. Soc.}, 351(9):3673--3686, 1999.

\bibitem[Bri93]{Brick93}
Stephen~G. Brick.
\newblock Quasi-isometries and ends of groups.
\newblock {\em J. Pure Appl. Algebra}, 86(1):23--33, 1993.

\bibitem[BZKL22]{BenZviKrophollerLyman22}
Michael Ben-Zvi, Robert Kropholler, and Rylee~Alanza Lyman.
\newblock Folding-like techniques for {CAT}(0) cube complexes.
\newblock {\em Math. Proc. Cambridge Philos. Soc.}, 173(1):227--238, 2022.

\bibitem[CN05]{ChatterjiNiblo05}
Indira Chatterji and Graham Niblo.
\newblock From wall spaces to {$\rm CAT(0)$} cube complexes.
\newblock {\em Internat. J. Algebra Comput.}, 15(5-6):875--885, 2005.

\bibitem[CS11]{CapraceSageev11}
Pierre-Emmanuel Caprace and Michah Sageev.
\newblock Rank rigidity for {CAT}(0) cube complexes.
\newblock {\em Geom. Funct. Anal.}, 21(4):851--891, 2011.

\bibitem[Dic80]{Dicks80}
Warren Dicks.
\newblock {\em Groups, trees and projective modules}, volume 790 of {\em
  Lecture Notes in Mathematics}.
\newblock Springer, Berlin, 1980.

\bibitem[DL21]{DaniLevcovitz21}
Pallavi Dani and Ivan Levcovitz.
\newblock Subgroups of right-angled {C}oxeter groups via {S}tallings-like
  techniques.
\newblock {\em J. Comb. Algebra}, 5(3):237--295, 2021.

\bibitem[Dun85]{Dunwoody85}
M.~J. Dunwoody.
\newblock The accessibility of finitely presented groups.
\newblock {\em Invent. Math.}, 81(3):449--457, 1985.

\bibitem[Geo08]{Geoghegan08}
Ross Geoghegan.
\newblock {\em Topological methods in group theory}, volume 243 of {\em
  Graduate Texts in Mathematics}.
\newblock Springer, New York, 2008.

\bibitem[GK91]{GeogheganKrasinkiewicz91}
Ross Geoghegan and J\'{o}zef Krasinkiewicz.
\newblock Empty components in strong shape theory.
\newblock {\em Topology Appl.}, 41(3):213--233, 1991.

\bibitem[GM85]{GeogheganMihalik85}
Ross Geoghegan and Michael~L. Mihalik.
\newblock Free abelian cohomology of groups and ends of universal covers.
\newblock {\em J. Pure Appl. Algebra}, 36(2):123--137, 1985.

\bibitem[GS19]{GeogheganSwenson19}
Ross Geoghegan and Eric Swenson.
\newblock On semistability of {$\rm CAT(0)$} groups.
\newblock {\em Groups Geom. Dyn.}, 13(2):695--705, 2019.

\bibitem[Hag22]{Hagen22}
Mark Hagen.
\newblock Large facing tuples and a strengthened sector lemma.
\newblock {\em Tunis. J. Math.}, 4(1):55--86, 2022.

\bibitem[Hou77]{Houghton77}
C.~H. Houghton.
\newblock Cohomology and the behaviour at infinity of finitely presented
  groups.
\newblock {\em J. London Math. Soc. (2)}, 15(3):465--471, 1977.

\bibitem[HT19]{HagenTouikan19}
Mark~F. Hagen and Nicholas W.~M. Touikan.
\newblock Panel collapse and its applications.
\newblock {\em Groups Geom. Dyn.}, 13(4):1285--1334, 2019.

\bibitem[HW08]{HaglundWise08}
Fr\'{e}d\'{e}ric Haglund and Daniel~T. Wise.
\newblock Special cube complexes.
\newblock {\em Geom. Funct. Anal.}, 17(5):1551--1620, 2008.

\bibitem[HW12]{HaglundWise12}
Fr\'{e}d\'{e}ric Haglund and Daniel~T. Wise.
\newblock A combination theorem for special cube complexes.
\newblock {\em Ann. of Math. (2)}, 176(3):1427--1482, 2012.

\bibitem[Kra77]{Krasinkiewicz77}
J\'{o}zef Krasinkiewicz.
\newblock Local connectedness and pointed {$1$}-movability.
\newblock {\em Bull. Acad. Polon. Sci. S\'{e}r. Sci. Math. Astronom. Phys.},
  25(12):1265--1269, 1977.

\bibitem[Lev98]{Levitt98}
Gilbert Levitt.
\newblock Non-nesting actions on real trees.
\newblock {\em Bull. London Math. Soc.}, 30(1):46--54, 1998.

\bibitem[Man20]{Manning20}
Jason Manning.
\newblock Cubulating spaces and groups, lecture notes.
\newblock
  http://pi.math.cornell.edu/~jfmanning/teaching/notes/cubulating20200303.pdf,
  2020.

\bibitem[Mih83]{Mihalik83}
Michael~L. Mihalik.
\newblock Semistability at the end of a group extension.
\newblock {\em Trans. Amer. Math. Soc.}, 277(1):307--321, 1983.

\bibitem[Mih96]{Mihalik96}
Michael~L. Mihalik.
\newblock Semistability of {A}rtin and {C}oxeter groups.
\newblock {\em J. Pure Appl. Algebra}, 111(1-3):205--211, 1996.

\bibitem[MT92a]{MihalikTschanz92a}
Michael~L. Mihalik and Steven~T. Tschantz.
\newblock One relator groups are semistable at infinity.
\newblock {\em Topology}, 31(4):801--804, 1992.

\bibitem[MT92b]{MihalikTschantz92b}
Michael~L. Mihalik and Steven~T. Tschantz.
\newblock Semistability of amalgamated products and {HNN}-extensions.
\newblock {\em Mem. Amer. Math. Soc.}, 98(471):vi+86, 1992.

\bibitem[Nic04]{Nica04}
Bogdan Nica.
\newblock Cubulating spaces with walls.
\newblock {\em Algebr. Geom. Topol.}, 4:297--309, 2004.

\bibitem[Rol16]{Roller16}
Martin Roller.
\newblock Poc sets, median algebras and group actions, 2016.

\bibitem[\'{S}06]{Swiatkowski06}
Jacek \'{S}wi\.{a}tkowski.
\newblock Regular path systems and (bi)automatic groups.
\newblock {\em Geom. Dedicata}, 118:23--48, 2006.

\bibitem[Sag95]{Sageev95}
Michah Sageev.
\newblock Ends of group pairs and non-positively curved cube complexes.
\newblock {\em Proc. London Math. Soc. (3)}, 71(3):585--617, 1995.

\bibitem[Sta71]{Stallings71}
John Stallings.
\newblock {\em Group theory and three-dimensional manifolds}, volume~4 of {\em
  Yale Mathematical Monographs}.
\newblock Yale University Press, New Haven, Conn.-London, 1971.
\newblock A James K. Whittemore Lecture in Mathematics given at Yale
  University, 1969.

\bibitem[SW05]{SageevWise05}
Michah Sageev and Daniel~T. Wise.
\newblock The {T}its alternative for {${\rm CAT}(0)$} cubical complexes.
\newblock {\em Bull. London Math. Soc.}, 37(5):706--710, 2005.

\bibitem[SW15]{SageevWise15}
Michah Sageev and Daniel~T. Wise.
\newblock Cores for quasiconvex actions.
\newblock {\em Proc. Amer. Math. Soc.}, 143(7):2731--2741, 2015.

\bibitem[SW22]{ShepherdWoodhouse22}
Sam Shepherd and Daniel~J. Woodhouse.
\newblock Quasi-isometric rigidity for graphs of virtually free groups with
  two-ended edge groups.
\newblock {\em J. Reine Angew. Math.}, 782:121--173, 2022.

\bibitem[Swa96]{Swarup96}
G.~A. Swarup.
\newblock On the cut point conjecture.
\newblock {\em Electron. Res. Announc. Amer. Math. Soc.}, 2(2):98--100, 1996.

\bibitem[Wis12]{WiseRiches}
Daniel~T. Wise.
\newblock {\em From riches to raags: 3-manifolds, right-angled {A}rtin groups,
  and cubical geometry}, volume 117 of {\em CBMS Regional Conference Series in
  Mathematics}.
\newblock Published for the Conference Board of the Mathematical Sciences,
  Washington, DC; by the American Mathematical Society, Providence, RI, 2012.

\end{thebibliography}

\end{document}